\documentclass[a4paper, 11pt]{article}
\usepackage[english]{babel}
\usepackage[width=16cm, top=2.5cm, bottom=2cm]{geometry}

\usepackage{tikz}
\usepackage{gauPDE}
\usepackage{amssymb}
\usepackage{amsmath}
\usepackage{stmaryrd}
\usepackage{empheq}
\usepackage[mathcal]{euscript}
\usepackage{mathrsfs}

\usepackage[shortlabels]{enumitem}
\setlist[enumerate, 1]{label=(\roman*)}
\setlist[enumerate, 2]{label=(\alph*)}
\setlist[enumerate, 3]{label=(\arabic*)}

\usepackage{amsthm}
\theoremstyle{plain}
\newtheorem{theorem}{Theorem}[section]
\theoremstyle{definition}
\newtheorem{definition}[theorem]{Definition}
\theoremstyle{plain}
\newtheorem{lemma}[theorem]{Lemma}
\theoremstyle{plain}
\newtheorem{proposition}[theorem]{Proposition}
\theoremstyle{plain}
\newtheorem{corollary}[theorem]{Corollary}
\theoremstyle{definition}
\newtheorem{example}[theorem]{Example}
\theoremstyle{remark}
\newtheorem{remark}[theorem]{Remark}
\theoremstyle{definition}
\newtheorem{assumption}[theorem]{Assumption}

\usepackage[backend=biber, style=alphabetic, giveninits, maxnames=6, maxalphanames=4, isbn=false, url=false, alldates=year]{biblatex}
\usepackage{csquotes}

\usepackage[colorlinks, linkcolor=blue, citecolor=blue!50!cyan]{hyperref}

\SetSymbolFont{stmry}{bold}{U}{stmry}{m}{n}
\allowdisplaybreaks

\newcommand{\ve}{\varepsilon}
\newcommand{\Om}{\Omega} 
\newcommand{\n}{\mathrm{n}} 
\newcommand{\calM}{\mathcal{M}} 
\newcommand{\calX}{\mathcal{X}} 
\newcommand{\rangeM}{\mathcal{C}} 
\newcommand{\HHn}{\HH_\n}

\newcommand{\boldXi}{\boldsymbol{\Xi}} 
\newcommand{\CE}{\mathcal{C}\mathcal{E}} 
\newcommand{\calW}{\mathcal{W}} 
\newcommand{\calD}{\mathcal{D}} 

\newcommand{\calF}{\mathcal{F}} 
\newcommand{\calU}{\mathcal{U}} 
\newcommand{\calV}{\mathcal{V}} 
\newcommand{\scrE}{\mathscr{E}} 
\newcommand{\scrR}{\mathscr{R}} 
\newcommand{\calE}{\mathcal{E}} 
\newcommand{\calEmin}{\mathcal{E}_\mathrm{min}} 
\newcommand{\Fconc}{F_\mathrm{conc}} 
\newcommand{\Fconv}{F_\mathrm{conv}} 
\newcommand{\Faff}{F_\mathrm{aff}} 

\newcommand{\boldN}{\boldsymbol{\mathrm{N}}} 
\newcommand{\boldu}{\boldsymbol{u}} 
\newcommand{\boldv}{\boldsymbol{v}} 
\newcommand{\boldj}{\boldsymbol{j}} 
\newcommand{\boldzeta}{\boldsymbol{\zeta}} 
\newcommand{\boldeta}{\boldsymbol{\eta}} 

\newcommand{\frakm}{\mathfrak{m}} 
\newcommand{\taumin}{\tau_\mathrm{min}} 
\newcommand{\taumax}{\tau_\mathrm{max}} 
\newcommand{\Amax}{A_*} 

\newcommand{\olbfu}{\overline{\boldsymbol{u}}} 
\newcommand{\olz}{\overline{z}}
\newcommand{\olu}{\overline{u}}
\newcommand{\ulbfu}{\underline{\boldsymbol{u}}} 
\newcommand{\ulz}{\underline{z}}
\newcommand{\ulu}{\underline{u}}
\newcommand{\whz}{\widehat{z}} 
\newcommand{\wtbfu}{\widetilde{\boldsymbol{u}}} 
\newcommand{\wtu}{\widetilde{u}}
\newcommand{\wtz}{\widetilde{z}}
\newcommand{\wtj}{\boldsymbol{\tilde{\jmath}}} 
\newcommand{\olj}{\boldsymbol{\bar{\jmath}}}
\newcommand{\kappaell}{{\kappa_\ell}}

\title{Dissipative solutions for viscoelastic phase separation with degenerate mobility: a dynamic transport approach}
\author{
	Moritz Immanuel Gau\footnote{Weierstrass Institute for Applied Analysis and Stochastics, Anton-Wilhelm-Amo-Str.\ 39, 10117 Berlin, Germany, \\\texttt{moritzimmanuel.gau@wias-berlin.de}; \;\texttt{katharina.hopf@wias-berlin.de}; \;\texttt{matthias.liero@wias-berlin.de}} \and Katharina Hopf$^*$ \and Matthias Liero$^*$
}
\numberwithin{equation}{section}
\date{}

\begin{document}
\maketitle

\begin{abstract}
	We study the convergence of minimising movement approximations to weak solutions of a fourth-order viscoelastic phase separation system with degenerate concave mobility.
	The underlying semi-implicit variational scheme combines a dynamic transport distance for the phase-field variable with a weighted $\LL^2$-distance for a transformed stress variable.
	We establish subsequential convergence of the scheme and show that the limiting solutions satisfy the energy-dissipation inequality.
	To this end, we derive a coercivity estimate for the metric slope, which yields the required space-time regularity for the De\,Giorgi interpolant, and construct constitutive fluxes whose dissipation is controlled by the squared metric slope.
	These estimates allow us to pass to the limit in the discrete energy-dissipation inequality.
	The weak formulation of the evolution equation for the phase-field variable is derived using the flow interchange technique.
	Under additional assumptions on the data, we further show that the limiting solutions satisfy an entropy-dissipation inequality.
\end{abstract}

\emph{Keywords: }Cahn–Hilliard equation, viscoelastic phase separation, gradient-flow structure, minimising movements scheme, degenerate mobility, generalized Wasserstein distance

\emph{MSC 2020: }
35K65, 
35A15, 
49Q22, 
35K41, 
49J45, 
35A01, 
76A10

\section{Introduction}

Viscoelastic phase separation (VPS) in polymer mixtures has attracted increasing attention in recent years; we refer to~\cite{Tana2000VPS,ZhouZhangE06,BDEHLS2021AVPS,Tana2022VPSB,CL2023PSRV} and the references therein.
It describes phase separation in binary fluids whose components evolve on different characteristic time and length scales.
In polymer solutions, this asymmetry reflects the contrast between solvent diffusion and polymer relaxation.
The resulting viscoelastic stresses influence the phase-separation process, which can feature a transient polymer network that later relaxes and breaks into smaller clusters.

For $d\geq 1$ let $\Omega\subset\R^d$ denote a bounded, smooth, convex domain.
We consider VPS models of the form
\begin{PDEsystem}\label{eq:VPSoriginal}
	\PDElineIn{\pa_tu}{-\divv(\boldj)}{(0,\infty)\times\Om},\\
	\PDElineIn{\pa_tq}{-\tfrac{1}{\tau(u)}q+A(u)\divv(\boldj)}{(0,\infty)\times\Om},\\
	\PDElineIn{\boldj}{-m(u)\nabla\bigparen{-\Delta u+f(u)-A(u)q}}{(0,\infty)\times\Om},
\end{PDEsystem}
which are variants of the purely dissipative models proposed in~\cite[Section III]{ZhouZhangE06} to describe VPS in polymeric mixtures.
The system is complemented by no-flux and homogeneous Neumann boundary conditions (see \eqref{eq:boundCondU}).
Here, $u:[0,\infty)\times\Om\to[0,\infty)$ denotes the phase-field variable (e.g.\ the concentration of polymers in a polymer solution) satisfying the physical constraints $0\leq u(t,x)\leq M$, where $M\in(0,+\infty]$.
In particular, the case $M=+\infty$ allows for unbounded phase-field variables, whereas finite $M$ corresponds to a volume fraction constraint.
The scalar variable $q:[0,\infty)\times\Om\to\R$ models the spherical part of the stress tensor.
Moreover, the mobility $m:\overline{(0,M)}\to\R_+$ is a non-negative concave function such that $m(0)=0$ and $m(M)=0$ if $M<+\infty$.
We write $f=F'$, where $F:\overline{(0,M)}\to\R$ is the free energy potential, while $A:\overline{(0,M)}\to\R$ and $\tau:\overline{(0,M)}\to\R_+$ are the bulk modulus and the relaxation time of the bulk stress, respectively, which couple the evolution of $u$ and $q$.
The system satisfies mass conservation $\int_\Om u(t,x)\dd x=\frakm \Leb^d(\Om)$ for some fixed $\frakm\in(0,M)$.

System~\eqref{eq:VPSoriginal} has a formal gradient-flow structure with driving energy
\[\scrE(u,q):=\int_\Om\Bigparen{\frac12\abs{\nabla u}^2+F(u)}\dd x+\frac12\int_\Om q^2\dd x.\]
The first term is the usual Cahn--Hilliard energy, while the second describes the contribution of the bulk stress.
The associated dual dissipation potential is
\[\scrR^*(u;\xi,\eta):=\frac12\int_\Om m(u)\abs{\nabla(\xi-A(u)\eta)}^2\dd x +\frac12\int_\Om\frac{\eta^2}{\tau(u)}\dd x.\]
With this choice, the system \eqref{eq:VPSoriginal} can formally be rewritten as 
\[\dot u = \D_{\xi}\scrR^*(u,-\D\scrE(u,q)) \quad\text{and}\quad \dot q = \D_{\eta}\scrR^*(u,-\D\scrE(u,q)).
\]
Formal sharp-interface asymptotics for related VPS models with a double-obstacle potential were studied in~\cite{HKMW_2026}, leading to non-local lower-order counterparts of the surface diffusion flow, with the gradient-flow structure of the diffuse-interface model reflected in that of the resulting geometric evolution law.

The purpose of the present work is to study the convergence of variational approximations for system~\eqref{eq:VPSoriginal} with a class of degenerate concave mobilities.
Our approach is based on a semi-implicit minimising movement scheme that combines dynamic transport for the phase-field variable with a weighted $\LL^2$-distance for a transformed stress variable.
A specific aim is to show that the limiting solutions satisfy the energy-dissipation inequality, ensuring their consistency with the thermodynamic structure of the model.

Our construction builds on two related works.
For degenerate Cahn--Hilliard and thin-film equations, Lisini, Matthes and Savar\'e~\cite{LisiniMatthesSavare2012} constructed weak solutions by minimising movements with respect to dynamic transport metrics associated with nonlinear mobilities.
A notable feature of this construction is that the physical bounds on the phase-field variable are preserved already at the discrete level.
For a simplified VPS model with constant mobility, a gradient-flow approach was recently developed by the first two authors in~\cite{GauHopf2025}, where a change of variables leads to an $\HH^{-1}\times \LL^2$-type Hilbertian gradient structure with an almost convexity property of the transformed energy.
This enabled a well-posedness theory for suitably regular initial data, including uniqueness and stability estimates, together with an energy-dissipation balance obtained by means of a chain rule.
In this paper, we use the same change of variables and a dynamic transport metric to account for the degenerate mobility.
A comparable geodesic convexity property for the Dirichlet energy is, however, not available in this transport setting, and a chain rule yielding the energy-dissipation balance does not appear to be known at the regularity of the weak solutions considered here.

To identify a metric structure that is well suited to the variational construction, we consider the change of variables from~\cite{GauHopf2025},
\[(u,q)=\Psi(u,z):=(u,z-K(u)),\qquad K(u):=\int_0^u A(r)\dd r,\]
and pull back the gradient structure of system~\eqref{eq:VPSoriginal} via $\Psi$.
Writing $\boldu=(u,z)$, the transformed energy and dual dissipation potential are given by
\begin{align*}
	\calE(\boldu)&:=\scrE(\Psi(\boldu))=\int_\Om\Bigparen{\frac12\abs{\nabla u}^2+F(u)}\dd x
	+\frac12\int_\Om(z-K(u))^2\dd x,\\
	\mathcal R^*(u;\tilde\xi,\tilde\eta)
	&:=\scrR^*(u;\tilde\xi+A(u)\tilde\eta,\tilde\eta)=\frac12\int_\Om\Bigparen{m(u)\abs{\nabla\tilde\xi}^2+\frac{\tilde\eta^2}{\tau(u)}}\dd x.
\end{align*}
Note that this change of variables renders the dual dissipation potential diagonal in the driving forces, at the expense of introducing a coupling between $u$ and $z$ in the energy.
The evolution equations in the new variables take the form
\begin{PDEsystem}\label{eq:evoTrans}
	\PDElineIn{\pa_tu}{-\divv\boldj}{(0,\infty)\times\Om},\\
	\PDElineIn{\pa_tz}{-\frac{1}{\tau(u)}(z-K(u))}{(0,\infty)\times\Om},\\
	\PDElineIn{\boldj}{-m(u)\nabla\bigl(-\Delta u+f(u)-A(u)(z-K(u))\bigr)}{(0,\infty)\times\Om}.
\end{PDEsystem}
This formulation exhibits the mixed parabolic--ODE structure of the system, coupling a fourth-order Cahn--Hilliard-type equation for $u$ to a relaxation equation for $z$.
In particular, there is no spatial regularisation mechanism in the equation for $z$.

The transport part of $\mathcal R^*$ is associated with a geodesic distance $\calW_m$ for the conserved variable $u$.
Such distances extend the Benamou--Brenier dynamic formulation of optimal transport~\cite{BenamouBrenier2000} to nonlinear concave mobilities and were introduced and developed in~\cite{DolbeaultNazaretSavare2009,CarrilloLisiniSavareSlepcev2010,LisiniMarigonda2010}.
We recall their definition and the properties needed here in Subsection~\ref{subsec:transportMetric}.
The $\calW_m$-minimising movement approach of~\cite{LisiniMatthesSavare2012} has been developed further for more general free energies~\cite{LoiblMatthesZinsl2016,Zinsl2017}, time-dependent mobilities~\cite{Zinsl2019} and systems~\cite{ZinslMatthes2015}.
Since no static Monge--Kantorovich formulation of $\calW_m$ in terms of transport maps or plans is known for general nonlinear
mobilities, the derivation of a weak PDE formulation is more delicate.
In the approach of~\cite{LisiniMatthesSavare2012}, the flow interchange technique introduced in~\cite{MatthesMcCannSavare2009} plays a crucial role not only in obtaining regularity estimates and compactness but also in showing that the limiting curves satisfy the evolution law.

In the present setting, the dependence of $\tau$ on $u$ suggests considering a fully coupled dynamic distance for $(u,z)$.
However, extending the flow interchange arguments to this distance would require establishing evolution variational inequalities for suitable auxiliary flows with respect to that distance (cf.~Section~\ref{ssec:flow-interchange}), which seems difficult in the coupled setting.
As in~\cite{GauHopf2025}, we therefore freeze the relaxation time $\tau$ at the preceding time step and work with the product metric 
\[\calD_v^2\bigparen{(u_0,z_0),(u_1,z_1)}:=\calW_m^2(u_0,u_1)+\bignorm{\sqrt{\tau(v)}(z_0-z_1)}_{\LL^2(\Om)}^2,\]
where $v$ is an admissible phase-field density.
For a time step $\kappa>0$, the resulting minimising movement scheme (MMS) reads
\[\boldu_\kappa^{n+1}\in\argmin_{\boldu\in\calX_\frakm^M\times\LL^2(\Om)}
\biggbraces{\frac{1}{2\kappa}\calD_{u_\kappa^n}^2(\boldu,\boldu_\kappa^n)+\calE(\boldu)},\]
where $\calX_\frakm^M$ denotes the set of densities with the prescribed mass and physical bounds; see Section~\ref{sec:setupAndStrategy}.
In contrast to the classical approach of~\cite{ElliottGarcke1996}, no regularisation of the mobility or the energy is introduced in the minimisation problem and, by construction, the discrete solutions preserve the prescribed mass and physical bounds at every time step.

Under the assumptions stated in Section~\ref{sec:setupAndStrategy}, our main result establishes subsequential convergence, as $\kappa\to0$, of the corresponding time interpolants to global weak solutions of~\eqref{eq:evoTrans} satisfying the energy-dissipation inequality
\[\calE(\boldu(t))+\int_0^t\int_\Om\Bigparen{\frac{\abs{\boldj}^2}{m(u)}
+\frac{\abs{z-K(u)}^2}{\tau(u)}}\dd x\dd s\leq\calE(\boldu^0).\]
Here, $\boldj$ is the constitutive diffusive flux appearing in both the continuity equation and the weak formulation of the last equation in~\eqref{eq:evoTrans}, and $\boldu^0$ denotes the initial state.
For the precise statement, including an entropy-dissipation inequality under additional assumptions, we refer to Theorem~\ref{th:mainTheorem}.
The energy-dissipation inequality reflects the formal gradient-flow structure of the model at the level of weak solutions and provides a starting point for further analysis, including the study of singular limits and weak--strong uniqueness.
Earlier variational constructions for fourth-order equations, such as~\cite{LisiniMatthesSavare2012,LoiblMatthesZinsl2016,Zinsl2017,CancesMatthesNabet2019, CancesMatthes2023, CarrilloEspositoFalcoFernandez2024}, yield weak solutions with non-increasing energy, but do not establish an energy-dissipation inequality of the above form.
To obtain such an inequality, we exploit the discrete energy-dissipation balance associated with the De\,Giorgi interpolant, which provides time-integrated control of the squared metric slope; an approach based on this interpolant is also considered in~\cite{KroemerLaux2022} for a Cahn--Hilliard model with linear mobility.
In the present setting, expressing the dissipation in this balance in terms of the constitutive fluxes and passing to the limit requires a suitable time-integrated $\HH^2(\Om)$-regularity of the phase-field component of the interpolant that is uniform in $\kappa$.
However, the telescoping argument used to obtain this regularity for the piecewise constant interpolant does not directly apply to the De\,Giorgi interpolant, and we instead derive a $\kappa$-uniform $\LL^2(0,T;\HH^2(\Om))$-bound for the phase-field component from a coercivity estimate for the metric slope.
With this regularity at hand, suitable variations in the definition of the slope allow us to construct constitutive fluxes for both interpolants and bound their dissipation by the corresponding squared metric slopes.
Combined with the discrete energy-dissipation balance, these estimates yield an inequality in terms of the constitutive fluxes, which the uniform regularity allows us to preserve in the limit.
An overview of this argument is given in Subsection~\ref{ssec:edi.overview}, and the main ingredients of the proof are outlined after Theorem~\ref{th:mainTheorem}.

Finally, let us mention that global weak solutions for a related VPS model with hydrodynamic coupling and stress diffusion were constructed by Brunk and Luk\'a\v{c}ov\'a-Medvid'ov\'a~\cite{BrunkLukacova2022} using regularisation and compactness arguments.
The present variational construction treats the purely dissipative system~\eqref{eq:VPSoriginal} without stress diffusion.

\paragraph{Outline.}
The paper is organised as follows.
In Section~\ref{sec:setupAndStrategy}, we introduce the minimising movement scheme, define the interpolants, state our main results and explain the strategy of the proofs.
Section~\ref{sec:preliminaries} contains a review of the dynamical transport metric and the flow interchange lemma.
Section~\ref{sec:discreteSolutions} is devoted to the proof of our first main result, which concerns the regularity of the time-discrete interpolants and the approximate system they satisfy.
Finally, we perform the limit passage in the time step parameter in Section~\ref{sec:mainTheorem} which concludes the proof of our second main result.

\section{Main Results}\label{sec:setupAndStrategy}

\subsection{Minimising movement scheme for VPS}
As highlighted in the introduction, for the rigorous analysis of the system in \eqref{eq:VPSoriginal}, it is advantageous to introduce a change of variables $q=z-K(u)$, with $K'=A$, and simplify the underlying variational structure. In particular, we rewrite the system~\eqref{eq:VPSoriginal} equivalently as the following fourth-order parabolic PDE coupled to an ODE, namely
\begin{PDEsystemTag}{\textsf{VPS}}\label{eq:VPS}
	\PDElineIn{\pa_tu}{-\divv(\boldj)}{(0,\infty)\times\Om},\\[0.5em]
	\PDElineIn{\pa_tz}{-\tfrac{1}{\tau(u)}(z{-}K(u))}{(0,\infty)\times\Om},\\[0.5em]
	\PDElineIn{\boldj}{-m(u)\nabla\bigparen{-\Delta u+f(u)-A(u)(z{-}K(u))}}{(0,\infty)\times\Om}.
\end{PDEsystemTag}
The system is complemented by no-flux and homogeneous Neumann boundary conditions, i.e.
\begin{equation}\label{eq:boundCondU}
	\boldj\cdot\n=0,\quad\nabla u\cdot\n=0\qquad\text{on }(0,\infty)\times\boundary{\Om},
\end{equation}
and initial conditions $u(0)=u^0$, $z(0)=z^0$. Here, $\n:\boundary{\Om}\to \mathbb{S}^{d-1}$ denotes the outer normal unit vector.

To set up the minimising movement scheme associated with~\eqref{eq:VPS}, we introduce the notation for the admissible measures and densities
\begin{align*}
	\rangeM&:=[0,M]\text{ if }M<+\infty,\qquad\rangeM:=[0,\infty)\text{ if }M=+\infty,\\
	\calM_\frakm^+(\varcl{\Om})&:=\bigset{\mu\in\calM^+(\varcl{\Om})}{\mu(\varcl{\Om})=\frakm\Leb^d(\Om)},\\
	\calX_\frakm^M&:=\bigset{u\in\LL^1(\Om)}{u\in\rangeM\text{ a.e.\ and }\textstyle\int_\Om u\dd x=\frakm\Leb^d(\Om)},\\
	\HHn^2(\Om)&:=\bigset{u\in\HH^2(\Om)}{\nabla u\cdot\n=0\text{ on }\boundary{\Om}}.
\end{align*}
The energy $\calE:\calX_\frakm^M\times\LL^2(\Om)\to \R_\infty:=\R\cup \{\infty\}$ is given via
\[\calE(\boldu)=\begin{cases}
	\int_\Om \bigparen{{\frac12\abs{\nabla u}^2+F(u)}}\dd x+\frac12\int_\Om(z-K(u))^2\dd x&\text{if }u\in\HH^1(\Omega),\\
	+\infty&\text{otherwise}.
\end{cases}\]
We now fix a time step $\kappa>0$ and choose the initial datum such that $\boldu^0=(u^0,z^0)\in\dom(\calE)$.
Then we construct iteratively $(\boldu_\kappa^n)_{n\in\N_0}$, starting from $\boldu_\kappa^0:=\boldu^0$ and letting
\begin{equation}\label{eq:TIM}
	\boldu_\kappa^{n+1}\in\argmin_{\boldu\in\calX_\frakm^M\times\LL^2(\Om)}\Psi_{\kappa,\boldu_\kappa^n}(\boldu) \quad\text{with}\quad \Psi_{\kappa,\boldu_\kappa^n}(\boldu):=\frac{1}{2\kappa}\calD_{u_\kappa^n}^2(\boldu,\boldu_\kappa^n)+\calE(\boldu).\tag{\textsf{MMS}}
\end{equation}
Here, the phase-dependent (pseudo-)metric $\calD_v$ is defined as
\[\calD_v(\boldu_1,\boldu_2):=\bigparen{\calW_m^2(u_1,u_2)+\bignorm{\sqrt{\tau(v)}(z_1-z_2)}_{\LL^2(\Om)}^2}^{1/2}\]
for $v\in\LL^2(\Om)$, $\boldu_1,\boldu_2\in\calX_\frakm^M\times\LL^2(\Om)$, and the distance $\calW_m$ denotes the dynamic transport metric associated with the mobility $m$. Its definition and properties used below are reviewed in Subsection~\ref{subsec:transportMetric}. The existence of minimisers in \eqref{eq:TIM} is ensured by Lemma~\ref{le:existenceTIM}.

As usual, we use these discrete values to define the piecewise constant forward and backward as well as affine interpolants $\olbfu_\kappa,\ulbfu_\kappa:[0,\infty)\to\calX_\frakm^M\times\LL^2(\Om)$, $\whz_\kappa:[0,\infty)\to\LL^2(\Om)$ as
\begin{equation}\label{eq:piecewiseInterpolants}
	\begin{alignedat}{1}
		&\olbfu_\kappa(t):=\boldu_\kappa^{\lceil t/\kappa\rceil}, \qquad \ulbfu_\kappa(t):=\boldu_\kappa^{\lfloor t/\kappa\rfloor}, \qquad \text{and}\\
		&\whz_\kappa(t):=\bigparen{\tfrac{t}{\kappa}-\bigl\lfloor\tfrac{t}{\kappa}\bigr\rfloor}z_\kappa^{\lceil t/\kappa\rceil}+\bigparen{1-\tfrac{t}{\kappa}+\bigl\lfloor\tfrac{t}{\kappa}\bigr\rfloor}z_\kappa^{\lfloor t/\kappa\rfloor}.
	\end{alignedat}
\end{equation}
In particular, $\olu_\kappa(n\kappa)=\ulu_\kappa(n\kappa)=u_\kappa^n$ and $\olz_\kappa(n\kappa)=\ulz_\kappa(n\kappa)=\whz_\kappa(n\kappa)=z_\kappa^n$ for every $n\in\N_0$. We mainly employ these interpolants to derive the time-discrete PDE.

Finally, as in \cite{ElliottGarcke1996, LisiniMatthesSavare2012}, we introduce the entropy functional $\calU:\calX_\frakm^M\to \R_+$ defined via
\begin{equation}\label{eq:defEntropy}
	\calU(u):=\int_\Omega U(u(x))\dd x,\quad\text{where}\quad U''(u)= \frac{1}{m(u)},\quad U(\frakm)=U'(\frakm)=0.
\end{equation}

\subsection{The energy-dissipation inequality}\label{ssec:edi.overview}
A central aim of our analysis is to show that solutions constructed by the above minimising movement scheme satisfy the energy-dissipation inequality 
\begin{align}\label{eq:edi.intro}
	\calE[u_t,z_t]+\int_0^t\int_\Om\Bigparen{\frac{\abs{\boldj}^2}{m(u)}+\frac{1}{\tau(u)}\abs{z-K(u)}^2}\dd x\dd s\leq\calE[u_0,z_0].\tag{\textsf{EDI}}
\end{align}
Due to the unavailability of a rigorous chain rule that would allow us to conclude~\eqref{eq:edi.intro} for typical weak solutions, we will establish this inequality using the variational nature of the time-discrete scheme. To this end, we introduce the variational interpolant, also known as the \emph{De\,Giorgi interpolant}, which is obtained by setting $\wtbfu_\kappa(n\kappa):=\boldu_\kappa^n$, and choosing, for $0<\theta<\kappa$ and $t:=n\kappa+\theta$,
\begin{equation}\label{eq:deGiorgiInterpolant}
	\wtbfu_\kappa(t)\in\argmin_{\boldu\in\calX_\frakm^M\times\LL^2(\Om)}\Bigparen{\frac{1}{2\theta}\calD_{u_\kappa^n}^2(\boldu,\boldu_\kappa^n)+\calE(\boldu)}.\tag{\textsf{DeG}}
\end{equation}
Note that, while minimisers of~\eqref{eq:deGiorgiInterpolant} may not be unique, we can invoke classical measurable choice theory (see e.g.~\cite[Theorem~6.9.6]{Bogachev_2007}) to ensure the existence of a Borel measurable selection $\wtbfu_\kappa:[0,\infty)\to\calX_\frakm^M\times \overline B_R^{\LL^2}$, for suitable $R=R(\boldu^0)<\infty$, where $\calX_\frakm^M$ is endowed with the strong $\LL^1$-topology and $\overline B_R^{\LL^2}$ with the weak $\LL^2$-topology. Since $\wtbfu_\kappa(t)\in \HH^1(\Om)\times \LL^2(\Om)$ for every $t\ge0$, the Lusin--Souslin theorem~\cite[Theorem~6.8.6]{Bogachev_2007} then implies Borel measurability in the strong topology of $\HH^1(\Om)\times\LL^2(\Om)$.
Fundamental properties of this interpolant are captured in Ambrosio et al.~\cite[Sections 3.1, 3.2]{AmbrosioGigliSavare2008}. In particular, the following discrete energy-dissipation balance is satisfied for every $n\in\N_0:$
\begin{equation}\label{eq:discreteEDB}
	\frac{\calD_{u_\kappa^n}^2(\boldu_\kappa^n,\boldu_\kappa^{n+1})}{2\kappa}+\int_0^\kappa\frac{\calD_{u_\kappa^n}^2(\boldu_\kappa^n,\wtbfu_\kappa(n\kappa+\theta))}{2\theta^2}\dd\theta=\calE(\boldu_\kappa^n)-\calE(\boldu_\kappa^{n+1}).
\end{equation}
By the slope inequality~\cite[Lemma 3.1.3]{AmbrosioGigliSavare2008}
\begin{equation}\label{eq:slopeEstimate}
	\abs{\pa\calE}_{\calD_{u_\kappa^n}}^2(\wtbfu_\kappa(t))\leq\frac{\calD_{u_\kappa^n}^2(\boldu_\kappa^n,\wtbfu_\kappa(t))}{\theta^2},
\end{equation}
the second term on the left-hand side of~\eqref{eq:discreteEDB} dominates the dissipation term involving the metric slope of $\calE$.
However, passing directly to the limit in the slope term in order to derive the abstract EDI, where dissipation is expressed in terms of metric speed and metric slope, would require a suitable lower semi-continuity property, which, to the best of our knowledge, is not available in the present setting.
In this paper, we therefore derive the version~\eqref{eq:edi.intro} of the energy-dissipation inequality, formulated in terms of the PDE fluxes.
We expect this formulation, combined with an entropy-dissipation inequality, to encode sufficient information for a weak--strong uniqueness principle.

\subsection{Data qualification}

The overall assumptions are the following:

\begin{assumption}[Domain]\label{as:domain}
	For the whole text, we assume that $\Om\subset\R^d$, $d\in\N$ is a bounded, convex domain (i.e.\ open and connected) that has a $\CC^l$-boundary for some $l\in\N$ satisfying
	\begin{equation}\label{eq:domainBoundaryRegularity}
		l\geq\frac{(d+2)^2+4}{2(d+2)}.
	\end{equation}
\end{assumption}

\begin{remark}
	Condition~\eqref{eq:domainBoundaryRegularity} is used for regularity results in Lemma~\ref{le:freeEnergyAlongHeatFlow}. In the most relevant cases $1\leq d\leq 3$, \eqref{eq:domainBoundaryRegularity} is satisfied for $l=3$.
\end{remark}

\begin{assumption}[Mobility]\label{as:mobility}
	Assume that $m\in\CC^0(\rangeM)\cap\CC^2((0,M))$ is concave, that $m>0$ on $(0,M)$, $m(0)=0$, and if $M<+\infty$ then additionally $m(M)=0$. Moreover, we assume
	\begin{equation}\label{eq:mobilityLipschitzConditions}
		\sup_{u\in(0,M)}\abs{m'(u)}<\infty\quad\text{and}\quad\sup_{u\in(0,M)}(-m''(u)m(u))<\infty.
	\end{equation}
\end{assumption}

In the case $M<+\infty$, we note that $m$ attains its global maximum in $(0,M)$. Concavity and non-negativity moreover imply
\begin{equation}\label{eq:mobilityGrowthEstimates}
	m(u)\geq\begin{cases}
		\frac{m(\frakm)}{\frakm}u&\text{if }u\leq\frakm,\\
		\frac{m(\frakm)}{M-\frakm}(M-u)&\text{if }u\geq\frakm\text{ and }M<+\infty,\\
		m(\frakm)&\text{if }u\geq\frakm\text{ and }M=+\infty.
	\end{cases}
\end{equation}

Besides positivity and concavity of $m$, which are required for $\calW_m$ to be well-defined, we also assume that $m$ is degenerate and satisfies the conditions~\eqref{eq:mobilityLipschitzConditions}, which imply in particular that $m$ vanishes linearly at the finite endpoints of $(0,M)$. Both properties play a crucial role in constructing the flow map associated with the so-called \emph{potential energy}, which we define later in~\eqref{eq:defPotentialEnergy}.

\begin{assumption}[Potential]\label{as:potential}
	We assume $F\in\CC^2(0,M)$ subject to
	\begin{subequations}\label{eq:potentialBounds}
		\begin{align}
			\abs{F(u)}&\leq C_1(1+u^\gamma),\label{eq:potentialUpperBound}\\
			m(u)F''(u)&\geq-c_1(1+m(u))\quad\text{for all }u\in(0,M)\label{eq:potentialLowerBound}
		\end{align}
	\end{subequations}
	for some constants $c_1,C_1\in(0,\infty)$ and an exponent $\gamma\in[1,\infty)$ with $\gamma\leq\frac{2d}{d-4}$ if $d\geq5$. With the \emph{pressure} $P:(0,M)\to\R$ defined by $P(u):=\int_\frakm^um(v)F''(v)\dd v$, we require a continuous extension $P\in\CC^0(\rangeM)$. In the case $M=+\infty$, we further impose the asymptotic decay
	\begin{equation}\label{eq:asymptoticBehaviorP}
		\lim_{u\to\infty}\frac{P(u)}{u^\omega+\abs{F(u)}}=0.
	\end{equation}
	for an exponent $\omega\in(2,\infty)$ satisfying $\omega\leq\min\{4\frac{d-1}{d-2},2\frac{d+2}{d-2}\}$ whenever $d\geq3$.
\end{assumption}

Observe that \eqref{eq:potentialLowerBound} can be simplified in the case $M<+\infty$ to $\inf_{u\in(0,M)}u(M{-}u)F''(u)>-\infty$, and \eqref{eq:potentialUpperBound} holds trivially.

Following~\cite[Subsection 2.4]{LisiniMatthesSavare2012}, $F$ decomposes into continuous convex, concave and affine parts $\Fconv,\Fconc,\Faff\in\CC^0(\rangeM)$, which satisfy for a constant $c_2\in(0,\infty)$:
\begin{equation}\label{eq:potentialDecompositionBounds}
	\Fconv(u)\geq\Fconv(\frakm)=0=\Fconc(\frakm)\geq\Fconc(u)\geq-c_2(1+u^2)\quad\text{for all }u\in\rangeM.
\end{equation}
Since $\Fconv$, $\Fconc$, $\Faff$ are continuous, this implies $F=\Fconv+\Faff+\Fconc\in\CC^0(\rangeM)$ as well.

We highlight that the constraints on $\omega$ guarantee the embedding
\begin{equation}\label{eq:interpolationExponentq}
	\LL^2(I;\HH^2(\Om))\cap\LL^\infty(I;\HH^1(\Om))\hookrightarrow\LL^\omega(I;\LL^\omega(\Om))
\end{equation}
for any compact interval $I\subset[0,\infty)$.

\begin{assumption}[Bulk modulus, relaxation coefficient]\label{as:coefficients}
	We assume $A,\tau\in\CC^0(\rangeM)$, and there are positive reals $\Amax$, $\taumax$, $\taumin$ such that
	\begin{align}
		\abs{A(u)}&\leq\Amax,\label{eq:assumptionA}\\
		\taumin\leq\tau(u)&\leq\taumax\qquad\text{for all }u\in\rangeM.\label{eq:assumptionTau}
	\end{align}
	In particular, the primitive $K(u):=\int_0^uA(r)\dd r$ is Lipschitz continuous.
\end{assumption}

\begin{remark}
	While condition~\eqref{eq:assumptionTau} is `optimal' for ensuring that the induced weight $\tau(u)$ yields an $\LL^2(\Om)$-scalar product equivalent to the standard one, it is possible to relax the boundedness~\eqref{eq:assumptionA} to allow more general polynomial growth. However, to carry out our analytical arguments, we would need to require the initial data $z^0$ to belong to some $\LL^q$-space with $q>2$, and some estimates would become more elaborate.

	Consider, for instance, the integral involving the decomposed term $A(u)(z-K(u))\Delta u$, which arises in the proof of Lemma~\ref{le:freeEnergyAlongHeatFlow}. One would need to estimate $A(u)K(u)\Delta u$ by a product of powers of $\HH^1$- and $\HH^2$-norms, so as to absorb the $\HH^2$-contribution with the appropriate exponent. In addition, for the mixed term $A(u)z\Delta u$, one would need to track the regularity of $z$ propagated from $z^0$ through the (discrete) ODE, similarly as in Step~3 of the proof of Theorem~\ref{th:mainTheorem}. All these estimates rely fundamentally on Sobolev embeddings for $\HH^1(\Om)$ and $\HH^2(\Om)$, leading to numerous dimension-dependent case distinctions. To streamline the presentation, we retain the uniform bound~\eqref{eq:assumptionA}.
\end{remark}

For the strong convergence of the variable $z$ and the limit passage in the entropy-dissipation inequality we make additional assumptions.

\begin{assumption}[Extra regularity]\label{as:entropyInequality}
	In addition to Assumptions~\ref{as:domain}, \ref{as:mobility}, \ref{as:potential} and~\ref{as:coefficients}, we impose
	\begin{enumerate}
		\item either $M<+\infty$ or ($M=+\infty$ and $d\leq3$).
		\item\label{it:potentialConcavePartVanishing} $\lim_{u\to a}m(u)\Fconc''(u)=0$ for $a\in\{0,M\}\setminus\{+\infty\}$.
		\item $\tau$ and $K/\tau$ are Lipschitz continuous.
	\end{enumerate}
\end{assumption}

\begin{example}[Typical choices for $m$, $F$, $A$ and $\tau$]
	In the case $M<+\infty$, the paradigmatic choice is the mobility
	\[m(u)=u(M-u)\]
	and the Flory--Huggins potential
	\[F(u)=\frac{1}{n_p}u\log(u)+\frac{1}{n_s}(M{-}u)\log(M{-}u)+\beta u(M{-}u),\]
	where $n_p,n_s>0$ are the molecular weights for the polymer and solvent molecules, and $\beta>0$ the Flory interaction parameter~\cite{ZhouZhangE06}. Semi-convex polynomial potentials such as the classical double-well potential
	\[F(u)=\frac14 u^2(M{-}u)^2\]
	also comply with Assumptions~\ref{as:potential}~and~\ref{as:entropyInequality}.
	To model dynamic asymmetry and an associated qualitative transition in the morphology of the polymeric structure, one may choose the bulk modulus $A$ and relaxation coefficient $\tau$ to be positive and monotonic functions of $u$. According to~\cite{ZhouZhangE06}, $A$ should be close to a minimum and maximum value in solvent rich and polymer rich regions characterised by $u\approx 0$ and $u\approx M$, respectively, with a rapid increase at a critical concentration $u_*\in(0,M)$ where polymers start to cross-link (e.g.\ $u_*=\frakm$). This leads, for instance, to
	\begin{align*}
		A(u)&=\frac{A_\mathrm{max}-A_\mathrm{min}}{2}\Bigparen{1+\tanh\Bigparen{\frac{\cot(\pi u_*/M)-\cot(\pi u/M)}{\ell}}}+A_\mathrm{min},\\
		\tau(u)&=\taumin+\frac{\taumax-\taumin}{M^2}u^2,
	\end{align*}
	where $0<A_\mathrm{min}<A_\mathrm{max}$, $0<\taumin\le\taumax$, and $0<\ell\ll1$.

	For $M=+\infty$, choosing a linear mobility $m(u)=u$ and a logarithmic potential with a concave quadratic perturbation
	\[F(u)=u\log(u)-\beta u^2,\qquad \beta>0,\]
	leads to thin-film equations. Note that singular concave potentials whose singularity is compensated by the mobility function are also admitted. This includes the example
	\[F(u)=-\beta u\log u,\qquad \beta>0,\]
	arising, for instance, in unstable thin-film equations studied in~\cite{LP_2002}.

	The degenerate Cahn--Hilliard equation from~\cite{LisiniMatthesSavare2012} can be obtained as a special case of our system by considering $A\equiv0$. Indeed, in this scenario, the $z$ variable decouples from the phase-field equation.
\end{example}

\subsection{Main results and strategy}

For vector-valued $\boldzeta\in\CC^1(\varcl{\Om};\R^d)$ and $\boldu=(u,z)\in\HH^2(\Om)\times\LL^2(\Om)$, we abbreviate
\begin{equation}\label{eq:defBoldN}
	\boldN[\boldu,\boldzeta]:=-\int_\Om\bigparen{\Delta u+A(u)(z-K(u))}\divv(m(u)\boldzeta)\dd x+\int_\Om P(u)\divv(\boldzeta)\dd x.
\end{equation}
Note that $(\psi,\boldzeta)\mapsto\int_0^\infty\psi\boldN[\boldu,\boldzeta]\dd t$ defines a linear, continuous map on $\CC_c^0((0,\infty))\times\CC^1(\varcl{\Om};\R^d)$ for $u\in\LL^\infty([0,\infty);\HH^1(\Om))\times\Lloc^2([0,\infty);\HH^2(\Om))$ and $z\in\Lloc^2([0,\infty);\LL^2(\Om))$. Indeed, we emphasise that in the case $M=+\infty$ the asymptotic condition~\eqref{eq:asymptoticBehaviorP} ensures $P(u)\in\Lloc^1([0,\infty);\LL^1(\Om))$.

\begin{theorem}[Regularity of discrete solutions and approximative system]\label{th:discreteSolutions}
	Assume $\boldu^0=(u^0,z^0)\in\dom(\calE)$ and let the interpolants $(\olu_\kappa,\ulu_\kappa,\wtu_\kappa)$, $(\olz_\kappa,\ulz_\kappa,\whz_\kappa,\wtz_\kappa)$ be defined as above.
	\begin{subequations}
		\begin{enumerate}
			\item The interpolants enjoy the regularity
			\begin{align*}
				\olu_\kappa,\ulu_\kappa(\cdot+\kappa),\wtu_\kappa&\in\LL^\infty([0,\infty);\HH^1(\Om))\cap\Lloc^2([0,\infty);\HHn^2(\Om)),\\
				\olz_\kappa,\ulz_\kappa,\whz_\kappa,\wtz_\kappa&\in\LL^\infty([0,\infty);\LL^2(\Om)),\quad\pa_t\whz_\kappa\in\LL^2([0,\infty);\LL^2(\Om)),
			\end{align*}
			All $\olu_\kappa,\ulu_\kappa,\wtu_\kappa$ stay in $\rangeM$ and have mean value $\frakm$ at any time $t\in[0,\infty)$.
			\item The following variational equations are fulfilled:
			The \emph{approximate equation for the order parameter}
			\begin{equation}\label{eq:approxEquationU}
				\int_0^\infty\int_\Om\psi'V\olu_\kappa\dd x\dd t=-\int_0^\infty\underline{\psi}_\kappa\boldN[\olbfu_\kappa,\nabla V]\dd t+O(\sqrt{\kappa})
			\end{equation}
			for all test functions $\psi\in\CC_c^1((0,\infty))$ and $V\in\CC^2(\varcl{\Om})$ with $\nabla V\cdot\n=0$ on $\boundary{\Om}$, where $\underline{\psi}_\kappa$ is the piecewise constant function defined as $\underline{\psi}_\kappa(t)=\psi\bigparen{\kappa\lfloor\frac{t}{\kappa}\rfloor}$, and the implicit constant depends on the data and the test functions $\psi$, $V$, but not on $\kappa$. Furthermore, it holds the \emph{discrete relaxation law}
			\begin{equation}\label{eq:approxEquationZ}
				\pa_t\whz_\kappa=-\frac{1}{\tau(\ulu_\kappa)}\bigparen{\olz_\kappa-K(\olu_\kappa)}\quad\text{in }\LL^2(\Om)
			\end{equation}
			for all $t\in[0,\infty)\setminus\set{n\kappa}{n\in\N}$.
			\item The \emph{discrete energy-dissipation inequality}
			\begin{align}\label{eq:approxFreeEnergyEDI}
				\calE(\olbfu_\kappa(t))+\int_s^t\int_\Om\biggbraces{&\frac12\Bigparen{\frac{\abs{\olj_\kappa}^2}{m(\olu_\kappa)}+\frac{1}{\tau(\ulu_\kappa)}\abs{\olz_\kappa-K(\olu_\kappa)}^2}\notag\\
					&+\frac12\Bigparen{\frac{\abs{\wtj_\kappa}^2}{m(\wtu_\kappa)}+\frac{1}{\tau(\ulu_\kappa)}\abs{\wtz_\kappa-K(\wtu_\kappa)}^2}}\dd x\dd r\leq\calE(\ulbfu_\kappa(s))
			\end{align}
			holds for all $0\leq s\leq t<\infty$, where $\olj_\kappa,\wtj_\kappa\in\LL^2([0,\infty);\LL^{4/3}(\Om))$ are determined by the \emph{constitutive flux relations}:
			\begin{equation}\label{eq:approxFluxConditions}
				\int_\Om\boldzeta\cdot\olj_\kappa(t)\dd x=\boldN[\olbfu_\kappa(t),\boldzeta],\qquad\int_\Om\boldzeta\cdot\wtj_\kappa(t)\dd x=\boldN[\wtbfu_\kappa(t),\boldzeta]
			\end{equation}
			for all $\boldzeta\in\CC^1(\varcl{\Om};\R^d)$ with $\boldzeta\cdot\n=0$ on $\boundary{\Om}$ and almost every $t\in[0,\infty)$.
			\item If Assumption~\ref{as:entropyInequality}~\ref{it:potentialConcavePartVanishing} holds true, then the \emph{discrete entropy-dissipation inequality} for the entropy functional defined in~\eqref{eq:defEntropy}
			\begin{equation}\label{eq:approxInternalEnergyEDI}
				\calU(\olu_\kappa(t))+\int_s^t\int_\Om\Bigparen{\abs{\Delta\olu_\kappa}^2+F''(\olu_\kappa)\abs{\nabla\olu_\kappa}^2+A(\olu_\kappa)(\olz_\kappa-K(\olu_\kappa))\Delta\olu_\kappa}\dd x\dd r\leq\calU(\ulu_\kappa(s)),
			\end{equation}
			is fulfilled for every $0\leq s\leq t<\infty$.
		\end{enumerate}
	\end{subequations}
\end{theorem}

\begin{remark}[Higher integrability for $\olj$ and $\wtj$]\label{re:discreteSolutions}
	Let $\alpha\in[0,1]$ satisfy $m(u)\lesssim1+u^\alpha$. We in addition have $\olj_\kappa,\wtj_\kappa\in\LL^2([0,\infty);\LL^\frac{2q}{\alpha+q}(\Om))$ for every $q\in[2,+\infty]$ satisfying $\HH^1(\Om)\hookrightarrow\LL^q(\Om)$. This follows from $\olu_\kappa,\wtu_\kappa\in\LL^\infty([0,\infty);\HH^1(\Om))$ and H\"older's inequality
	\begin{align*}
		\int_0^\infty\Bigparen{\int_\Om\abs{\olj_\kappa}^\frac{2q}{\alpha+q}\dd x}^\frac{\alpha+q}{q}\dd t&\leq\int_0^\infty\Bigparen{\int_\Om\abs{m(\olu_\kappa)}^\frac{q}{\alpha}\dd x}^\frac{\alpha}{q}\Bigparen{\int_\Om\frac{\abs{\olj_\kappa}^2}{m(\olu_\kappa)}\dd x}\dd t\\
		&\leq2^{1-\alpha}\sup_{t\in[0,\infty)}\norm{1+\olu_\kappa(t)}_{\LL^q(\Om)}^\alpha\int_0^\infty\int_\Om\frac{\abs{\olj_\kappa}^2}{m(\olu_\kappa)}\dd x\dd t<\infty.
	\end{align*}
\end{remark}

\begin{theorem}\label{th:mainTheorem}
	For every sequence $(\kappaell)_{\ell\in\N}\downarrow0$, there exists a subsequence (non-relabelled) and a triple $(u,z,\boldj)$ with the regularity
	\begin{subequations}\label{eq:regularities}
		\begin{align}
			u&\in\CC_{w,b}([0,\infty);\HH^1(\Om))\cap\Lloc^2([0,\infty);\HH^2(\Om)),\label{eq:regularityU}\\
			z&\in\LL^\infty([0,\infty);\LL^2(\Om))\text{ with }\pa_tz\in\LL^2([0,\infty);\LL^2(\Om)),\label{eq:regularityZ}\\
			\boldj&\in\LL^2([0,\infty);\LL^{4/3}(\Om;\R^d)),\label{eq:regularityFluxes}
		\end{align}
	\end{subequations}
	such that the interpolants from Theorem~\ref{th:discreteSolutions} converge in the sense as specified in Proposition~\ref{pr:convergentSubsequences}, and the following properties hold true:
	\begin{subequations}\label{eq:solutionNotion}
		\begin{enumerate}
			\item \emph{Continuity equation} for the phase-field variable, i.e.\ for every $\eta\in\CC_c^1((0,\infty)\times\varcl{\Om})$
			\begin{equation}\label{eq:equationU}
				\int_0^\infty\int_\Om u\pa_t\eta\dd x\dd t=-\int_0^\infty\int_\Om\boldj\cdot\nabla\eta\dd x\dd t,\qquad u(0)=u^0,
			\end{equation}
			as well as the \emph{constitutive equation for the diffusive flux for a.e.\ $t>0$}
			\begin{equation}\label{eq:fluxCondition}
				\int_\Om\boldj(t)\cdot\boldzeta\dd x=\boldN[\boldu(t),\boldzeta]
			\end{equation}
			for every $\boldzeta\in\CC_c^1(\varcl{\Om};\R^d)$ with $\boldzeta\cdot\n=0$ on $\boundary{\Om}$.\\[1mm]
			\emph{Relaxation law for $z$}
			\begin{equation}\label{eq:equationZ}
				\pa_tz=-\frac{1}{\tau(u)}\bigparen{z-K(u)}\quad\text{in }\LL^2([0,\infty);\LL^2(\Om)),\qquad z(0)=z^0.
			\end{equation}
			\item The \emph{energy-dissipation inequality} for every $t\in[0,\infty)$
			\begin{equation}\label{eq:freeEnergyEDI}
				\calE(\boldu(t))+\int_0^t\int_\Om\Bigparen{\frac{\abs{\boldj}^2}{m(u)}+\frac{1}{\tau(u)}\abs{z-K(u)}^2}\dd x\dd s\leq\calE(\boldu^0),
			\end{equation}
			\item If Assumption~\ref{as:entropyInequality} and $z^0\in\LL^\infty(\Om)$ are satisfied in addition, then $z\in\Lloc^\infty([0,\infty);\LL^\infty(\Om))$, the convergence~\eqref{eq:convergenceCwIL2z} in $z$ holds strongly,
			and we have the \emph{entropy-dissipation inequality} for every $t\in[0,\infty)$
			\begin{equation}\label{eq:internalEnergyEDI}
				\calU(\boldu(t))+\int_0^t\int_\Om\Bigparen{\abs{\Delta u}^2+F''(u)\abs{\nabla u}^2+A(u)(z-K(u))\Delta u}\dd x\dd s\leq\calU(\boldu^0).
			\end{equation}
		\end{enumerate}
	\end{subequations}
\end{theorem}

\begin{remark}
	\leavevmode
	\begin{enumerate}
		\item A posteriori, the time derivative $\pa_tu$ exists in $\LL^2([0,\infty);(\WW^{1,4}(\Om))^*)$, namely
		\[\product{\pa_tu(t)}{V}_{(\WW^{1,4}(\Om))^*,\WW^{1,4}(\Om)}:=\int_\Om\boldj(t)\cdot\nabla V\dd x.\]
		According to Remark~\ref{re:discreteSolutions}, the regularity for $\boldj$ improves if $m$ grows less than linearly. In particular, in the case $M<+\infty$ we always have $\pa_tu\in\LL^2([0,\infty);(\HH^1(\Om))^*)$.
		\item Helly's selection theorem tells us that, up to a further subsequence, the energies along the piecewise constant interpolants $\calE(\olbfu_\kappa)$ and $\calE(\ulbfu_\kappa)$ converge pointwise in time to non-increasing functions $\overline{E}$ and $\underline{E}$, respectively, and $\overline{E}=\underline{E}$ almost everywhere. Without Assumption~\ref{as:entropyInequality} we can only expect weak convergence of $\olz_\kappa$ and $\ulz_\kappa$, allowing to conclude $\calE(\boldu(t))\leq\overline{E}(t)\leq\underline{E}(t)$. If Assumption~\ref{as:entropyInequality} holds true and~\eqref{eq:potentialUpperBound} is satisfied for $\gamma\leq\frac{2d}{d-2}$ in case $d\geq3$, we also have $\calE(\olbfu_\kappa(t)),\calE(\ulbfu_\kappa(t))\to\calE(\boldu(t))$ for almost every $t\in[0,\infty)$, in particular, $\calE(\boldu)=\overline{E}=\underline{E}$ almost everywhere.
	\end{enumerate}
\end{remark}

\paragraph{Key ingredients.}
A crucial point of our analysis is the derivation of the energy-dissipation inequality~\eqref{eq:edi.intro} (cf.~\eqref{eq:freeEnergyEDI}).
To this end, we exploit the discrete energy-dissipation balance~\eqref{eq:discreteEDB} in two ways: first, for inferring uniform bounds in $\LL^2(0,T;\HH^2(\Om))$ for the De\,Giorgi interpolant $\wtu_\kappa$ and, second, for expressing the dissipation in terms of the PDE fluxes.
While for the piecewise constant interpolant, the bound in $\LL^2(0,T;\HH^2(\Om))$ follows from flow interchange with the entropy $\calU$ and a telescoping structure, for the De\,Giorgi interpolant we take advantage of the metric slope term in the discrete energy-dissipation balance, for which we establish a coercivity property of the form
\[\norm{\Delta u}^2_{\LL^2(\Om)}\lesssim 1+\calE(\boldu)+\abs{\pa\calE}_{\calD_w}^2(\boldu),\qquad \boldu=(u,z),\]
cf.\ Proposition~\ref{pr:aPrioriH2DeGiorgiPrelim}.
This is achieved by considering variations induced by the Neumann heat flow emanating from $u$ in the definition of $\abs{\pa\calE}_{\calD_w}(\boldu)$. A key ingredient in this step is the Lions--Villani square root estimate~\cite{LV_1995}
\begin{align}\label{eq:square-root.est}
	\norm{\nabla\sqrt{v}}_{\LL^4(\Om)}^2\lesssim \norm{\Delta v}_{\LL^2(\Om)}
\end{align}
for non-negative $v\in\HHn^2(\Om)$, see also~\cite[(3.4)]{GianazzaSavareToscani2009}, which, combined with the corresponding estimate for $M-v$ if $M<+\infty$, controls the $\calW_m$-action along the heat flow, since the non-negative concave function $m$ vanishes at most linearly at its points of degeneracy (cf.~\eqref{eq:mobilityGrowthEstimates}).
In order to express the dissipation in terms of the constitutive fluxes, we use variations along a diffusively regularised nonlinear drift equation in the definition of the metric slope. The resulting estimate shows that the first variation of $\calE$ defines a bounded linear functional on a product of weighted $\LL^2$-spaces. Riesz representation then provides the fluxes $\olj_\kappa$ and $\wtj_\kappa$, together with bounds on their dissipation in terms of the corresponding metric slopes.
At the same time, the slope inequality~\eqref{eq:slopeEstimate}, for the piecewise constant interpolant applied at $t=\kappa(n+1)$, bounds both dissipation terms in~\eqref{eq:discreteEDB} from below by the corresponding squared metric slopes of the piecewise constant and De\,Giorgi interpolants. Combining these estimates with the flux bounds yields the discrete energy-dissipation inequality~\eqref{eq:approxFreeEnergyEDI}.

The entropy-dissipation inequality~\eqref{eq:internalEnergyEDI} is obtained by a more classical argument using flow interchange with the entropy functional $\calU$. After carefully handling the possible singularities of $F''$, this first yields the discrete inequality~\eqref{eq:approxInternalEnergyEDI}. An additional difficulty specific to the VPS system is the passage to the limit in the coupling term $A(u)(z{-}K(u))\Delta u$. Under additional assumptions on the data, we use the discrete relaxation equation to derive uniform $\LL^\infty((0,T)\times\Om)$-bounds for the $z$-interpolants. A stability estimate then yields their strong convergence in $\LL^\infty(0,T;\LL^2(\Om))$. This allows us to pass to the limit in the coupling term and conclude~\eqref{eq:internalEnergyEDI}.

Finally, let us mention that the derivation of the weak formulation of the evolution equation for $u$ requires a separate argument, again based on flow interchange, but now involving entropically regularised potential energies. This step mostly follows~\cite{LisiniMatthesSavare2012} and, in particular, requires the mobility function $m$ to vanish linearly at the finite endpoints of $(0,M)$.

\section{Preliminaries}\label{sec:preliminaries}

\subsection{Transport distance with concave mobility}\label{subsec:transportMetric}

In the following, we review the definition and properties of the generalised Benamou--Brenier distance from~\cite{DolbeaultNazaretSavare2009, CarrilloLisiniSavareSlepcev2010, LisiniMarigonda2010}. To this end, let $\Om\subset\R^d$ be a convex domain, and $T\in(0,\infty)$. We denote by $\CE_\frakm(0,T)$ the set of all pairs $(\mu,\boldXi)=(\mu_t,\boldXi_t)_{t\in(0,T)}$ of Borel families of measures in $\calM_\frakm^+(\varcl{\Om})\times\calM(\varcl{\Om};\R^d)$ satisfying
\[\mu\in\CC_w([0,T];\calM(\varcl{\Om})),\qquad\int_0^T\norm{\boldXi_t}_\mathrm{TV}\dd t<+\infty,\]
as well as the continuity equation
\[\begin{alignedat}{2}
	\pa_t\mu_t+\divv(\boldXi_t)&=0\quad&\text{in }&(0,T)\times\Om,\\
	\boldXi_t\cdot\n&=0\quad&\text{on }&(0,T)\times\boundary{\Om}
\end{alignedat}\]
in the distributional sense, i.e.\ for every $\varphi\in\CC^1([0,T]\times\varcl{\Om})$ and $t_0<t_1$ it holds
\[\int_{\varcl{\Om}}\varphi_{t_1}\dd\mu_{t_1}-\int_{\varcl{\Om}}\varphi_{t_0}\dd\mu_{t_0}=\int_{t_0}^{t_1}\int_{\varcl{\Om}}\pa_t\varphi_t\dd\mu_t\dd t+\int_{t_0}^{t_1}\int_{\varcl{\Om}}\nabla\varphi_t\cdot\dd\boldXi_t\dd t.\]
Then, for $\mu^0,\mu^1\in\calM_\frakm^+(\varcl{\Om})$ fixed, we let $\CE_\frakm(0,T;\mu^0\to\mu^1)$ be the subset of $\CE_\frakm(0,T)$ whose elements $(\mu,\boldXi)$ satisfy in addition $\mu_0=\mu^0$ and $\mu_T=\mu^1$.

Next, we define the so-called \emph{action functional} $\Phi:\calM_\frakm^+(\varcl{\Om})\times\calM(\varcl{\Om};\R^d)\to[0,+\infty]$. To this end, we abbreviate
\[\phi:[0,\infty)\times\R^d\to[0,\infty],\qquad\phi(u,\boldj):=\begin{cases}
	\frac{\abs{\boldj}^2}{m(u)}&\text{if }u\in\rangeM,\\
	+\infty&\text{otherwise},
\end{cases}\]
where we used the convention $\frac00:=0$. Then, if $\mu=u\Leb^d\vert_\Om+\mu^\perp$, $\boldXi=\boldj\Leb^d\vert_\Om+\boldXi^\perp$ denote the Lebesgue decompositions, $\nu:=\mu^\perp+\abs{\boldXi^\perp}$, $\mu^\perp=u^\perp\nu$ and $\boldXi^\perp=\boldj^\perp\nu$, we put
\begin{equation}\label{eq:defActionFunctional}
	\Phi(\mu,\boldXi):=\int_\Om\phi(u,\boldj)\dd\Leb^d(x)+\int_{\varcl{\Om}}\phi^\infty(u^\perp,\boldj^\perp)\dd\nu(x).
\end{equation}
Here, $\phi^\infty$ denotes the recession function of $\phi$, i.e.\ $\phi^\infty(v,\boldzeta):=\lim_{\lambda\nearrow+\infty}\lambda\phi(\lambda v,\boldzeta)$.

In some circumstances the expression~\eqref{eq:defActionFunctional} simplifies. For instance, if $\Phi(\mu,\boldXi)<+\infty$ and $\mu\ll\Leb^d\vert_{\varcl{\Om}}$, it automatically follows $\boldXi\ll\Leb^d\vert_{\varcl{\Om}}$ (cf.~\cite[Lemma 3.7]{DolbeaultNazaretSavare2009}) which implies that the recession integral vanishes.
Another special case occurs when $M<+\infty$. Since $\phi^\infty(v,\boldzeta)=\infty$ whenever $v>0$ in this scenario, $\Phi(\mu,\boldXi)<+\infty$ always entails $\mu\ll\Leb^d\vert_{\varcl{\Om}}$, whence $\boldXi\ll\Leb^d\vert_{\varcl{\Om}}$ as well, and the recession integral is zero again.

The generalised Benamou--Brenier pseudo-distance $\calW_m$ is defined on $\calM_\frakm^+(\varcl{\Om})$ as
\begin{equation}\label{eq:defMetric}
	\calW_m(\mu^0,\mu^1):=\inf\biggset{\biggparen{T\int_0^T\Phi(\mu_t,\boldXi_t)\dd t}^{1/2}}{(\mu,\boldXi)\in\CE_\frakm(0,T;\mu^0\to\mu^1)}.
\end{equation}
The definition of $\calW_m$ is independent of the choice of $T$ and, when restricted to $\calX^M_{\mathfrak m}$ if $M<+\infty$, satisfies all metric axioms, except that it may take the value $+\infty$.

On a bounded convex domain and in the case $0\leq u\leq M<\infty$, convergence in $\calW_m$ is equivalent, under the assumptions on $m$ specified above, to narrow convergence of the associated measures $u\Leb^d$, see \cite[Theorem 6]{LisiniMarigonda2010}. Since the densities are uniformly bounded, this is also equivalent to weak-$*$ convergence in $\LL^\infty(\Omega)$.
When $M=+\infty$, convergence in $\calW_m$ implies narrow convergence of the associated measures $u\Leb^d$, under the assumptions on $m$ stated above. The converse is in general not true: without a uniform bound on the densities, narrow convergence alone need not imply convergence in $\calW_m$, cf.~\cite[Section 5.1]{DolbeaultNazaretSavare2009}.

The following properties of $\calW_m$ are relevant for our analysis:

\begin{lemma}\label{le:metricProperties}
	\leavevmode
	\begin{enumerate}
		\item\label{it:metricPropertyConvergence} If $\mu_n,\mu\in\calM_\frakm^+(\varcl{\Om})$, then
		\[\lim_{n\to\infty}\calW_m(\mu_n,\mu)=0\qquad\Longrightarrow\qquad\mu_n\starto\mu\text{ in }\calM(\varcl{\Om}).\]
		\item\label{it:metricPropertyWeak*LSC} If $\mu_n,\nu_n,\mu,\nu\in\calM_\frakm^+(\varcl{\Om})$ such that $\mu_n\starto\mu$ and $\nu_n\starto\nu$ in $\calM(\varcl{\Om})$, then
		\[\calW_m(\mu,\nu)\leq\liminf_{n\to\infty}\calW_m(\mu_n,\nu_n).\]
		\item\label{it:metricPropertySublevels} The sublevel sets of the form
		\[S_C^\calE:=\bigset{u\in\calX_\frakm^M}{\exists z\in\LL^2(\Om)\text{ such that }\calE[u,z]\leq C}\]
		for $C\in(0,\infty)$ are compact in $(\calX_\frakm^M,\calW_m)$.
	\end{enumerate}
\end{lemma}

For assertions~\ref{it:metricPropertyConvergence} and~\ref{it:metricPropertyWeak*LSC} we refer to~\cite[Theorems 5.5, 5.6]{DolbeaultNazaretSavare2009} (for $M=+\infty$) and~\cite[Theorem 2 (b), (c)]{LisiniMarigonda2010} (if $M<+\infty$).
To see~\ref{it:metricPropertySublevels} for $M=+\infty$, we first note that $S_C^\calE$ is compact in $\LL^2(\Om)$ by~\eqref{eq:freeEnergyLowerBound} and Sobolev embedding. Moreover, by the estimates~\eqref{eq:mobilityGrowthEstimates} for $m$, it is easy to see that, 
\[\calW_m(u_0,u_1)\lesssim\calW_2(u_0,u_1)+\norm{u_0-u_1}_{\dot\WW^{-1,2}(\Om)},\]
where $\calW_2$ denotes the classical $\LL^2$-Monge--Kantorovich distance. Since $\calW_2$ is equivalent to the narrow convergence on the compact set $\varcl{\Om}$, this yields the desired compactness property of $S_C^\calE$.

\subsection[The flow interchange technique and \texorpdfstring{$\lambda$}{lambda}-flows]{The flow interchange technique and \texorpdfstring{$\boldsymbol{\lambda}$}{lambda}-flows}\label{ssec:flow-interchange}

The flow interchange method, introduced by Matthes et al.~\cite{MatthesMcCannSavare2009}, compares a minimising movement step with the short-time evolution generated by an auxiliary functional $\calF$. Under suitable assumptions on this auxiliary flow, the comparison yields a discrete estimate for the dissipation of the energy along that flow; summing the estimate over the time steps gives a priori bounds. In the present work, an adapted flow interchange argument for our semi-implicit scheme provides key regularity estimates for the discrete approximants and helps deriving the associated time-discrete equations. It does not, however, directly yield the uniform $\LL^2(0,T;\HH^2(\Om))$-bound for the De\,Giorgi interpolant; we obtain that bound instead from a coercivity property of the energy slope.

\begin{definition}[$\lambda$-flows]
	Let $\calF:\calX_\frakm^M\to(-\infty,+\infty]$ be a proper, lower semicontinuous functional.
	\begin{enumerate}
		\item A \emph{semigroup} on a subset $X_0\subseteq\calX_\frakm^M$ is a family of maps $S^t:X_0\to X_0$, parametrised by $t\in[0,\infty)$, such that $S^0=\mathrm{id}_{X_0}$ and $S^{t+s}(u)=S^t(S^s(u))$ for every $u\in X_0$ and $s,t\geq0$.
		\item We call a semigroup $(S^t)_{t\geq0}$ on $\dom(\calF)$ a \emph{$\lambda$-flow for $\calF$} for some $\lambda\in\R$ if it satisfies the \emph{Evolutionary Variational Inequality} (EVI)
		\begin{align}\label{eq:EVI}
			\frac12\limsup_{t\searrow0}\frac{\calW_m^2(S^t(u),v)-\calW_m^2(u,v)}{t}+\frac{\lambda}{2}\calW_m^2(u,v)+\calF(u)\leq\calF(v)\tag{\textsf{EVI}}
		\end{align}
		for every $u,v\in\dom(\calF)$ with $\calW_m(u,v)<+\infty$.
	\end{enumerate}
\end{definition}

Testing \eqref{eq:EVI} with $v:=u$, one obtains $\lim_{t\searrow0}\calW_m(S^t(u),u)=0$, i.e.\ every $\lambda$-flow for $\calF$ is a continuous semigroup. This property, together with the triangle inequality, ensures that the `limsup' in \eqref{eq:EVI} takes a finite value for all $v\in\dom(\calF)$ with $\calW_m(u,v)<+\infty$.\\

Two prominent examples of such flows relevant in this work are the heat flow, associated with the internal energy~\eqref{eq:defInternalEnergy}, and the nonlinear drift-diffusion flow, corresponding to the regularised potential energy~\eqref{eq:defPotentialEnergyReg}. In the present setting, the flow interchange lemma takes the following form:

\begin{lemma}\label{le:flowInterchange}
	Let $\calF:\calX_\frakm^M\to(-\infty,+\infty]$ be proper, lower semicontinuous and $(S_\calF^t)_{t\geq0}$ a $\lambda$-flow for $\calF$ for some $\lambda\in\R$. Moreover, let $(\boldu_\kappa^n)_{n\in\N}$ be the minimisers of \eqref{eq:TIM}. Whenever $u_\kappa^n\in\dom(\calF)$ then
	\begin{equation}\label{eq:flowInterchangeEstimate}
		\calF(u_\kappa^n)-\calF(u_\kappa^{n-1})\leq\kappa\liminf_{s\searrow0}\Bigparen{\frac{\calE[S_\calF^s(u_\kappa^n),z_\kappa^n]-\calE[u_\kappa^n,z_\kappa^n]}{s}}-\frac{\lambda}{2}\calW_m^2(u_\kappa^n,u_\kappa^{n-1}).
	\end{equation}
\end{lemma}

\begin{proof}
	Using the minimising property $\Psi_{\kappa,\boldu_\kappa^{n-1}}(\boldu_\kappa^n)\leq\Psi_{\kappa,\boldu_\kappa^{n-1}}[S_\calF^s(u_\kappa^n),z_\kappa^n]$, we infer
	\begin{multline*}
		\frac{1}{2\kappa}\calW_m^2(u_\kappa^n,u_\kappa^{n-1})+\frac{1}{2\kappa}\Bignorm{\sqrt{\tau(u_\kappa^{n-1})}(z_\kappa^n-z_\kappa^{n-1})}_{\LL^2(\Om)}^2+\calE(\boldu_\kappa^n)\\
		\leq\frac{1}{2\kappa}\calW_m^2(S_\calF^s(u_\kappa^n),u_\kappa^{n-1})+\frac{1}{2\kappa}\Bignorm{\sqrt{\tau(u_\kappa^{n-1})}(z_\kappa^n-z_\kappa^{n-1})}_{\LL^2(\Om)}^2+\calE[S_\calF^s(u_\kappa^n),z_\kappa^n].
	\end{multline*}
	Observing that the norm terms on both sides cancel each other out, and multiplying both sides by $\kappa/s$, we are left with
	\[\frac{1}{2s}\Bigparen{\calW_m^2(u_\kappa^n,u_\kappa^{n-1})-\calW_m^2(S_\calF^s(u_\kappa^n),u_\kappa^{n-1})}\leq\kappa\frac{\calE[S_\calF^s(u_\kappa^n),z_\kappa^n]-\calE(\boldu_\kappa^n)}{s}.\]
	Thus, \eqref{eq:flowInterchangeEstimate} follows by taking the limes inferior $s\searrow0$ and applying \eqref{eq:EVI} to $u=u_\kappa^n$ and $v=u_\kappa^{n-1}$.
\end{proof}

We will apply Lemma~\ref{le:flowInterchange} later to the internal energy $\calF=\calU$ in Subsection~\ref{subsec:H2regularity} to show the $\HH^2$-regularity, and to the regularised potential energy $\calF=\calV_\ve$ in Subsection~\ref{subsec:discreteSystem} to derive the PDE.

\paragraph{The entropy functional.}

The \emph{entropy} functional on $\calX_\frakm^M$ is defined as in \eqref{eq:defEntropy}, i.e.,
\begin{equation}\label{eq:defInternalEnergy}
	\calU(v):=\int_\Om U(v(x))\dd x\in[0,+\infty],
\end{equation}
where $U:\rangeM\to[0,\infty)$ denotes the (uniquely determined) convex function that satisfies $U''=1/m$ and $U(\frakm)=U'(\frakm)=0$, i.e.\
\[U(v):=\int_\frakm^v\frac{v-s}{m(s)}\dd s.\]

Referring to \cite[Lemma 2.3]{LisiniMatthesSavare2012}, $\calU$ is lower semicontinuous, we have $\LL^2(\Om)\cap\calX_\frakm^M\subseteq\dom(\calU)$
and there exists $C_\calU\in(0,\infty)$, depending only on $\Om$, $\frakm$ and $m(\frakm)$, such that
\begin{equation}\label{eq:internalEnergyUpperBound}
	0\leq\calU(v)\leq C_\calU(1+\norm{v}_{\LL^2(\Om)}^2)\quad\text{for all }v\in\calX_\frakm^M.
\end{equation}
It is not difficult to see that the PDE associated with $\calU$ is formally given by the heat equation with homogeneous Neumann boundary conditions
\begin{PDEsystemTag}{\textsf{HE}}\label{eq:heatEquation}
	\PDElineIn{\pa_tv}{\Delta v}{(0,\infty)\times\Om},\\
	\PDElineOn{\nabla v\cdot\n}{0}{(0,\infty)\times\boundary{\Om}},\\
	\PDElineIn{v(0)}{v^0}{\Om}.
\end{PDEsystemTag}

However, proving that $\calU$ indeed generates a $\lambda$-flow $(S_\calU^t)_{t\geq0}$ with trajectories $t\mapsto S_\calU^t(v^0)$ solving~\eqref{eq:heatEquation} weakly for initial data $v^0\in\LL^2(\Om)\cap\calX_\frakm^M$ is rather subtle. Detailed proofs on the existence of a 0-flow for $(\calU,\calW_m)$ based on the Eulerian calculus method can be found in~\cite[Lemma 4.2]{LisiniMatthesSavare2012} and~\cite[Sections 5, 6]{CarrilloLisiniSavareSlepcev2010}.

\paragraph{The potential energy.}

For a potential $V\in\CC^2(\varcl{\Om})$ satisfying $\nabla V\cdot\n=0$ on $\boundary{\Om}$, we define the corresponding \emph{potential energy}
\begin{equation}\label{eq:defPotentialEnergy}
	\calV(v):=\calV_V(v):=\int_\Om V(x)v(x)\dd x.
\end{equation}
It is obvious that $\calV$ is lower semicontinuous in $(\calX_\frakm^M,\calW_m)$. Unfortunately, as pointed out in~\cite[Subsection 2.3]{CarrilloLisiniSavareSlepcev2010}, it is impossible to show that $\calV$ generates a $\lambda$-flow unless $m$ is linear. As a workaround, we therefore follow~\cite{LisiniMatthesSavare2012} and resort to the entropic regularisation
\begin{equation}\label{eq:defPotentialEnergyReg}
	\calV_\ve(v):=\calV(v)+\ve\calU(v).
\end{equation}
Let us briefly explain why this regularisation is useful. For a geodesic $\gamma:[0,1]\to\calX_\frakm^M$ with associated tangential potential $\varphi$, one can formally show
\begin{align*}
	\frac{\dd^2}{\dd s^2}\calV_\ve(\gamma)&\geq\int_\Om m(\gamma)m''(\gamma)\Bigparen{(\nabla\varphi\cdot\nabla\gamma)(\nabla\varphi\cdot\nabla V)-\frac12(\nabla\gamma\cdot\nabla V)\abs{\nabla\varphi}^2}\dd x\\
	&\quad+\int_\Om m(\gamma)m'(\gamma)(\Hess V)\nabla\varphi\cdot\nabla\varphi\dd x-\frac{\ve}{2}\int_\Om m''(\gamma)\abs{\nabla\gamma}^2\abs{\nabla\varphi}^2\dd x\\
	&\overset{!}{\geq}\lambda_\ve\int_\Om m(\gamma)\abs{\nabla\varphi}^2\dd x,
\end{align*}
where the `!'-indicated inequality corresponds to geodesic $\lambda_\ve$-convexity. For $\ve=0$, there are counterexamples showing that no single choice of $\lambda_0$ works uniformly among all geodesics. If, however, $\ve>0$, the desired inequality can be shown via Young's inequality and the assumptions on $m'$ and $m''m$ from~\eqref{eq:mobilityLipschitzConditions}, with
\[\lambda_\ve=-\norm{m'}_\infty\norm{V}_{\CC^2(\Om)}-\frac{9}{8\ve}\norm{V}_{\CC^1(\Om)}^2\norm{m''m}_\infty.\]
As for the heat flow, the existence of a $\lambda_\ve$-flow generated by $\calV_\ve$ is rigorously justified by the Eulerian calculus method, and for initial data $v^0\in\LL^2(\Om)\cap\calX_\frakm^M$, the trajectory $v_t:=S_{\calV_\ve}^t(v^0)$ is a weak solution to the nonlinear drift-diffusion equation
\begin{PDEsystemTag}{\textsf{nDDE}}\label{eq:viscousPotentialEquation}
	\PDElineIn{\pa_tv}{\divv(m(v)\boldzeta)+\ve\Delta v}{(0,\infty)\times\Om},\\
	\PDElineOn{\nabla v\cdot\n}{0}{(0,\infty)\times\boundary{\Om}},\\
	\PDElineIn{v(0)}{v^0}{\Om}
\end{PDEsystemTag}
with the drift field given by $\boldzeta=\nabla V$. We point out that the argument requires the additional conditions~\eqref{eq:mobilityLipschitzConditions} on $m$ as well as the degeneracy $m(0)=0$ (and $m(M)=0$) to ensure that $v$ stays in $\rangeM$. For a full proof we refer the interested reader to~\cite[Proposition 4.6]{LisiniMatthesSavare2012}.

\section{Analysis of the time-discrete scheme}\label{sec:discreteSolutions}

In this section, we prove Theorem~\ref{th:discreteSolutions} by establishing uniform estimates, deriving $\HH^2$-regularity, and proving the approximate system and discrete energy-dissipation inequality.

\subsection{Existence of minimisers and uniform bounds}\label{subsec:existenceOfMinimisers}

We begin by proving existence of minimisers for the scheme in~\eqref{eq:TIM} at each time step and deriving the uniform bounds needed for the subsequent analysis.

\begin{lemma}[Properties of $\calE$]\label{le:freeEnergyProperties}
	The functional $\calE$ is lower semicontinuous on $\calX_\frakm^M\times\LL^2(\Om)$ with respect to the weak* convergence in $\calM(\varcl{\Om})$ and the weak convergence in $\LL^2(\Om)$. Moreover, there exists a constant $C_\calE\in(0,\infty)$, depending only on $\Om$, $\frakm$, $d$ and $c_2$ (see \eqref{eq:potentialDecompositionBounds}), such that
	\begin{equation}\label{eq:freeEnergyLowerBound}
		\frac18\norm{u}_{\HH^1(\Om)}^2+\int_\Om\Fconv(u)\dd x+\frac12\norm{z-K(u)}_{\LL^2(\Om)}^2\leq\calE(\boldu)+C_\calE
	\end{equation}
	for every $\boldu=(u,z)\in\calX_\frakm^M\times\LL^2(\Om)$. In particular, $(\calX_\frakm^M\cap\HH^2(\Om))\times\LL^2(\Om)\subseteq\dom(\calE)\subseteq(\calX_\frakm^M\cap\HH^1(\Om))\times\LL^2(\Om)$ and $\calEmin:=\inf\calE(\boldu)<+\infty$.
\end{lemma}

\begin{proof}
	The energy estimate~\eqref{eq:freeEnergyLowerBound} can be shown precisely as in~\cite[Lemma 2.4]{LisiniMatthesSavare2012} by using \eqref{eq:potentialDecompositionBounds} and the interpolation estimate
	\[\norm{u}_{\LL^2(\Om)}\leq C\norm{\nabla u}_{\LL^2(\Om)}^{d/(d+2)}\norm{u}_{\LL^1(\Om)}^{2/(d+2)}+C\norm{u}_{\LL^1(\Om)}\]
	for a constant $C=C(\Om,d)>0$.
	For lower semicontinuity, we let $\boldu_n,\boldu\subset\calX_\frakm^M\times\LL^2(\Om)$, $n\in\N$ satisfy $u_n\starto u$ in $\calM(\varcl{\Om})$ and $z_n\weakto z$ in $\LL^2(\Om)$ as $n\to\infty$. It is not restrictive to assume $E:=\liminf_{n\to\infty}\calE(\boldu_n)<+\infty$ and $u_n\in\HH^1(\Om)$, $E=\lim_{n\to\infty}\calE(\boldu_n)$ and $\calE(\boldu_n)\leq E+1$ for every $n\in\N$. From~\eqref{eq:freeEnergyLowerBound} we get $u\in\HH^1(\Om)$, $u_n\weakto u$ in $\HH^1(\Om)$, $u_n\to u$, $K(u_n)\to K(u)$ in $\LL^2(\Om)$ and $u_n\to u$ pointwise almost everywhere. Therefore, lower semicontinuity for $\calE$ follows from weak lower semicontinuity of the norms, Fatou's lemma and~\eqref{eq:potentialDecompositionBounds}.
\end{proof}

The next lemma tells us that minimisers of~\eqref{eq:TIM} and~\eqref{eq:deGiorgiInterpolant} exist.

\begin{lemma}\label{le:existenceTIM}
	Fix $\boldv=(v,p)\in\dom(\calE)$, $w\in\calX_\frakm^M$ and $\kappa>0$. Then the minimisation problem
	\begin{equation}
		\min_{\boldu}\Psi_{\kappa,\boldv}^w(\boldu),\quad\text{where }\Psi_{\kappa,\boldv}^w(\boldu):=\frac{1}{2\kappa}\calD_w^2(\boldu,\boldv)+\calE(\boldu)
	\end{equation}
	among $\boldu=(u,z)\in\calX_\frakm^M\times\LL^2(\Om)$ has a solution that attains a finite value.
\end{lemma}

\begin{proof}[Proof (via direct method).]
	Since $\calEmin\leq E:=\inf\Psi_{\kappa,\boldv}^w\leq\Psi_{\kappa,\boldv}^w(\boldv)=\calE(\boldv)<\infty$, the infimum is finite.
	In particular, we find an infimising sequence $(\boldu_n)_{n\in\N}\subset\calX_\frakm^M\times\LL^2(\Om)$ such that $E=\lim_{n\to\infty}\Psi_{\kappa,\boldv}^w(\boldu_n)$ and $\sup_{n\in\N}\Psi_{\kappa,\boldv}^w(\boldu_n)<\infty$. From the lower bound of $\calE$ in~\eqref{eq:freeEnergyLowerBound}, we obtain a limit $\boldu=(u,z)\in\HH^1(\Om)\times\LL^2(\Om)$ such that $u_n\weakto u$ in $\HH^1(\Om)$ and $z_n\weakto z$ in $\LL^2(\Om)$, in particular $u_n\starto u$ in $\calM(\varcl{\Om})$. The lower semicontinuity of $\calE$ from Lemma~\ref{le:freeEnergyProperties} and lower semicontinuity of the weighted $\LL^2$-norm and of $\calW_m$ from Lemma~\ref{le:metricProperties}~\ref{it:metricPropertyWeak*LSC} ensure that $\Psi_{\kappa,\boldv}^w(\boldu)=E$.
\end{proof}

In the setting of the previous lemma, the minimising property always implies
\begin{equation}\label{eq:TIMminimisingProperty}
	\calE(\boldu)+\frac{1}{2\kappa}\calD_w^2(\boldu,\boldv)\leq\calE(\boldv).
\end{equation}
In particular, if we insert the minimisers $(\boldu_\kappa^n)_{n\in\N}$ from scheme~\eqref{eq:TIM}, we have via telescoping
\begin{equation}\label{eq:TIMTelescopeEstimate}
	\calE(\boldu_\kappa^n)+\frac{1}{2\kappa}\sum_{k=1}^n\calD_{u_\kappa^{k-1}}^2(\boldu_\kappa^{k-1},\boldu_\kappa^k)\leq\calE(\boldu^0).
\end{equation}
Let $\olu_\kappa$, $\ulu_\kappa$, $\wtu_\kappa$, and $\olz_\kappa$, $\ulz_\kappa$, $\whz_\kappa$, $\wtz_\kappa$ be defined as in~\eqref{eq:piecewiseInterpolants} and~\eqref{eq:deGiorgiInterpolant}. We now derive the basic uniform estimates for these interpolants from~\eqref{eq:TIMTelescopeEstimate}, using the coercivity of $\calE$.

\begin{lemma}[Basic uniform estimates]\label{le:aPrioriEasy}
	There exists a constant $C_2\in(0,\infty)$, only depending on $C_\calE$, $\calEmin$, $\calE(\boldu^0)$, $\taumin$ and $\Amax$, such that
	\begin{subequations}
		\begin{align}
			\norm{\olu_\kappa(t)}_{\HH^1(\Om)}+\norm{\ulu_\kappa(t)}_{\HH^1(\Om)}+\norm{\wtu_\kappa(t)}_{\HH^1(\Om)}&\leq C_2\label{eq:aPrioriLIH1u}\\
			\calW_m(\olu_\kappa(t),\ulu_\kappa(t))+\calW_m(\olu_\kappa(t),\wtu_\kappa(t))&\leq C_2\sqrt{\kappa}\label{eq:aPrioriInterpolantsU}\\
			\calW_m(\olu_\kappa(s),\olu_\kappa(t))&\leq C_2\sqrt{\kappa+\abs{s-t}}\label{eq:aPrioriCauchyWu}\\
			\norm{\olz_\kappa(t)}_{\LL^2(\Om)}+\norm{\ulz_\kappa(t)}_{\LL^2(\Om)}+\norm{\whz_\kappa(t)}_{\LL^2(\Om)}+\norm{\wtz_\kappa(t)}_{\LL^2(\Om)}&\leq C_2\label{eq:aPrioriLIL2z}\\
			\norm{\olz_\kappa(t)-\ulz_\kappa(t)}_{\LL^2(\Om)}+\norm{\olz_\kappa(t)-\whz_\kappa(t)}_{\LL^2(\Om)}+\norm{\olz_\kappa(t)-\wtz_\kappa(t)}_{\LL^2(\Om)}&\leq C_2\sqrt{\kappa}\label{eq:aPrioriInterpolantsZ}\\
			\norm{\dot\whz_\kappa}_{\LL^2([0,\infty);\LL^2(\Om))}&\leq C_2\label{eq:aPrioriL2L2dotz}
		\end{align}
	\end{subequations}
	for every $s,t\in[0,\infty)$.
\end{lemma}

\begin{proof}
	Fix $\kappa>0$, $n\in\N_0$, $0<\theta\leq\kappa$ and put $t:=n\kappa+\theta$. From Lemma~\ref{le:freeEnergyProperties}, \eqref{eq:TIMminimisingProperty} and the telescope estimate~\eqref{eq:TIMTelescopeEstimate} we obtain
	\[
		\frac18\norm{\wtu_\kappa(t)}_{\HH^1(\Om)}^2+\frac12\norm{\wtz_\kappa(t)-K(\wtu_\kappa(t))}_{\LL^2(\Om)}^2-C_\calE\leq\calE(\wtbfu_\kappa(t))\leq\calE(\boldu_\kappa^n)\leq\calE(\boldu^0).
	\]
	The estimates \eqref{eq:aPrioriLIH1u} and \eqref{eq:aPrioriLIL2z} easily follow from there. Now let additionally $\ell\in\N_0$, $0<\vartheta\leq\kappa$, $s:=\ell\kappa+\vartheta$. Then, using \eqref{eq:TIMTelescopeEstimate} once again,
	\begin{align*}
		\calW_m^2(\olu_\kappa(s),\olu_\kappa(t))&=\calW_m^2(u_\kappa^{\ell+1},u_\kappa^{n+1})\leq\biggparen{\sum_{k=\ell\wedge n+1}^{\ell\vee n}\calW_m(u_\kappa^k,u_\kappa^{k+1})}^2\\
		&\leq\abs{\ell-n}\sum_{k=\ell\wedge n+1}^{\ell\vee n}\calW_m^2(u_\kappa^k,u_\kappa^{k+1})\leq2\kappa\abs{\ell-n}(\calE(\boldu^0)-\calEmin).
	\end{align*}
	Noting that $\abs{\ell-n}\leq1+\frac{1}{\kappa}\abs{s-t}$, \eqref{eq:aPrioriCauchyWu} is established. Similarly we show \eqref{eq:aPrioriL2L2dotz}, namely
	\[\norm{\dot{\whz}_\kappa}_{\LL^2(0,t;\LL^2(\Om))}^2\leq\frac{1}{\taumin}\sum_{k=1}^{n+1}\kappa\Bignorm{\sqrt{\tau(u_\kappa^{k-1})}\frac{z_\kappa^k-z_\kappa^{k-1}}{\kappa}}_{\LL^2(\Om)}^2\leq\frac{2}{\taumin}(\calE(\boldu^0)-\calEmin).\]
	The estimates \eqref{eq:aPrioriInterpolantsU} and \eqref{eq:aPrioriInterpolantsZ} are easily obtained from~\eqref{eq:TIMminimisingProperty}, e.g., \[\calW_m(\ulu_\kappa(t),\wtu_\kappa(t))\leq\sqrt{\theta}\sqrt{2(\calE(\boldu^0)-\calEmin)}\] and $\norm{\olz_\kappa(t)-\whz_\kappa(t)}_{\LL^2(\Om)}\leq(1-\frac{\theta}{\kappa})\sqrt{\kappa}\sqrt{\frac{2(\calE(\boldu^0)-\calEmin)}{\taumin}}$.
\end{proof}

\subsection[\texorpdfstring{$\HH^2(\Om)$}{H2} regularity for \texorpdfstring{$u$}{u}]{\texorpdfstring{$\boldsymbol{\HH^2(\Om)}$}{H2} regularity for \texorpdfstring{$u$}{u}}\label{subsec:H2regularity}

We now establish $\HH^2(\Omega)$-regularity for the phase-field interpolants, treating the piecewise constant and De\,Giorgi interpolants separately.
The following auxiliary result is a consequence of~\cite[Theorem 3.1]{GianazzaSavareToscani2009}.

\begin{lemma}\label{le:nabla2mEstimate}
	There exists a constant $C_3\in(0,\infty)$, depending only on $\Om$, $d$, $M$, $\frakm$ and $m(\frakm)$, such that
	\begin{equation}\label{eq:nabla2mEstimate}
		\Bignorm{\frac{\abs{\nabla v}^2}{m(v)}}_{\LL^q(\Om)}\leq C_3\bigparen{\norm{\Delta v}_{\LL^2(\Om)}+\norm{\nabla v}_{\LL^{2q}(\Om;\R^d)}^2}
	\end{equation}
	for every $v\in\HHn^2(\Om)$ and $q\in[1,2]$, with convention $\frac{0}{0}:=0$.
\end{lemma}

\begin{proof}
	The estimates for $m$ captured in \eqref{eq:mobilityGrowthEstimates} yield
	\[\frac{m(\frakm)}{m(u)}\leq\begin{cases}
		\frac{\frakm}{u}+1&\text{for }M=+\infty,\\
		\frac{\frakm}{u}+\frac{M-\frakm}{M-u}&\text{for }M<+\infty
	\end{cases}\]
	for all $u\in\rangeM$. Then, using $\frac{\abs{\nabla v}^2}{v}=4\abs{\nabla\sqrt{v}}^2$ and inequality $16\int_\Om\abs{\nabla\sqrt{v}}^4\dd x\leq(2d+8)\int_\Om\abs{\Delta v}^2\dd x$ (cf.~\cite[Lemma A.1]{LisiniMatthesSavare2012} and \cite[Theorem 3.1, eqn. (3.4)]{GianazzaSavareToscani2009}), we obtain, in the case $M=+\infty$,
	\begin{align*}
		\Bignorm{\frac{\abs{\nabla v}^2}{m(v)}}_{\LL^q(\Om)}&\leq\frac{\frakm}{m(\frakm)}\Bignorm{\frac{\abs{\nabla v}^2}{v}}_{\LL^q(\Om)}+\frac{1}{m(\frakm)}\norm{\nabla v}_{\LL^{2q}(\Om)}^2\\
		&\leq\frac{\frakm}{m(\frakm)}\Bignorm{2\nabla\sqrt{v}}_{\LL^{2q}(\Om)}^2+\frac{1}{m(\frakm)}\norm{\nabla v}_{\LL^{2q}(\Om)}^2\\
		&\leq\frac{\frakm}{m(\frakm)}(\Leb^d(\Om))^\frac{2-q}{2q}\sqrt{2d-8}\norm{\Delta v}_{\LL^2(\Om)}+\frac{1}{m(\frakm)}\norm{\nabla v}_{\LL^{2q}(\Om)}^2.
	\end{align*}
	For $M<+\infty$, an analogous argument, applied to both $\frac{\abs{\nabla v}^2}{v}$ and $\frac{\nabla(M-v)^2}{M-v}$, shows
	\[\Bignorm{\frac{\abs{\nabla v}^2}{m(v)}}_{\LL^q(\Om)}\leq\frac{M}{m(\frakm)}(\Leb^d(\Om))^\frac{2-q}{2q}\sqrt{2d-8}\norm{\Delta v}_{\LL^2(\Om)}.\qedhere\]
\end{proof}

To take advantage of the flow interchange lemma, we next identify the dissipation of $\calE$ along the heat flow.

\begin{lemma}\label{le:freeEnergyAlongHeatFlow}
	Let $v$ be the solution to the heat equation~\eqref{eq:heatEquation} with homogeneous Neumann condition and some initial value $v^0\in\calX_\frakm^M\cap\HH^1(\Om)$. Moreover, let $z\in\LL^2(\Om)$. If
	\begin{equation}\label{eq:freeEnergyLiminfAlongHeatFlowBounded}
		\liminf_{s\searrow0}\frac{1}{s}\bigparen{\calE[v_s,z]-\calE[v^0,z]}>-\infty,
	\end{equation}
	then $v^0\in\HHn^2(\Om)$, and we have the estimate
	\begin{equation}\label{eq:freeEnergyLiminfAlongHeatFlowQualitativeEstimate}
		-\liminf_{s\searrow0}\frac{1}{s}\bigparen{\calE[v_s,z]-\calE[v^0,z]}\geq\frac{1}{2}\int_\Om\abs{\Delta v^0}^2\dd x-C_4(1+\calE[v^0,z])
	\end{equation}
	for a constant $C_4\in(0,\infty)$ depending only on $C_3$ (from Lemma~\ref{le:nabla2mEstimate}), $C_\calE$, $c_1$ and $\Amax$. If Assumption~\ref{as:entropyInequality}~\ref{it:potentialConcavePartVanishing} holds in addition, we have the sharp inequality
	\begin{multline}\label{eq:freeEnergyLiminfAlongHeatFlowRefinedEstimate}
		-\liminf_{s\searrow0}\frac{1}{s}\bigparen{\calE[v_s,z]-\calE[v^0,z]}\\
		\geq\int_\Om\abs{\Delta v^0}^2\dd x+\int_\Om F''(v^0)\abs{\nabla v^0}^2\dd x+\int_\Om(z-K(v^0))A(v^0)\Delta v^0\dd x.
	\end{multline}
\end{lemma}

\begin{proof}
	We follow the ideas of Lisini et al.~\cite[Lemma 4.3]{LisiniMatthesSavare2012} and \cite[Lemma 4.4]{MatthesMcCannSavare2009}.
	The overall strategy is to prove first that $t\mapsto\calE[v_t,z]$ is differentiable in $(0,\infty)$ with derivative
	\begin{align}\label{eq:freeEnergyDerivativeAlongHeatFlow}
		\ddd{t}\calE[v_t,z]&=-\int_\Om\abs{\Delta v_t}^2\dd x-\int_\Om F''(v_t)\abs{\nabla v_t}^2\dd x-\int_\Om(z-K(v_t))A(v_t)\Delta v_t\dd x\notag\\
		&=:-I_1(t)-I_2(t)-I_3(t),
	\end{align}
	and, secondly, that $\calE[v,z]$ is continuous at $t=0$. Hence, Lagrange's mean value theorem is applicable and provides $\xi(s)\in(0,s)$ for every $s>0$ such that
	\begin{equation*}
		\frac{1}{s}\bigparen{\calE[v_s,z]-\calE[v^0,z]}=\ddd{t}\bigg\vert_{t=\xi(s)}\calE[v,z].
	\end{equation*}
	Inserting \eqref{eq:freeEnergyDerivativeAlongHeatFlow} and \eqref{eq:freeEnergyLiminfAlongHeatFlowBounded} and exploiting the assumptions on $F$ and $A$, one finds a sequence $(v_{\xi(s_n)})_{n\in\N}$ being uniformly bounded in $\HH^2(\Om)$. Since $v_s\to v^0$ in $\HH^1(\Om)$ already, we deduce $v_{s_n}\weakto v^0$ in $\HH^2(\Om)$ and $v^0\in\HH^2(\Om)$. We do this rigorously now.\\

	\noindent
	\textbf{Step 1:} Derivation of \eqref{eq:freeEnergyDerivativeAlongHeatFlow} and estimate~\eqref{eq:freeEnergyDerivativeAlongHeatFlowEstimate} below.\\\noindent
	Since $\Om$ has a $\CC^l$-boundary with $l\geq2$, it follows $v\in\CC^\infty((0,\infty);\HH^2(\Om))$.
	To achieve formula~\eqref{eq:freeEnergyDerivativeAlongHeatFlow}, we differentiate each term of $\calE$ individually.
	We easily obtain $I_1$ and $I_3$ by differentiating the Dirichlet part and coupling term, relying on the Lipschitz continuity of $K$.
	The integral involving $F$ can be differentiated due to the strong maximum principle Lemma~\ref{le:strongMaximumPrinciple}, which ensures $v\in\CC^0((0,\infty)\times\varcl{\Om})$ and $0<\inf_{x\in\bar\Om}v(t,x)\leq\sup_{x\in\bar\Om}v(t,x)<M$ for positive $t>0$. Hence, $v$ is uniformly separated from possible singularities of $F$. For this comparison result, we relied on the boundary regularity~\eqref{eq:domainBoundaryRegularity}.

	Next, using the lower bound~\eqref{eq:potentialLowerBound} for $F$ and Lemma~\ref{le:nabla2mEstimate}, one finds for every $\ve>0$
	\begin{align*}
		-\int_\Om F''(v_t)\abs{\nabla v_t}^2\dd x&\leq c_1\biggparen{\int_\Om\abs{\nabla v_t}^2\dd x+\int_\Om\frac{\abs{\nabla v_t}^2}{m(v_t)}\dd x}\\
		&\leq c_1\biggparen{(1+C_3)\int_\Om\abs{\nabla v_t}^2\dd x+C_3\norm{\Delta v_t}_{\LL^2(\Om)}}\\
		&\leq c_1(1+C_3)\int_\Om\abs{\nabla v_t}^2\dd x+\ve\norm{\Delta v_t}_{\LL^2(\Om)}^2+\frac{c_1^2C_3^2}{4\ve}.
	\end{align*}
	The term $I_3$ is estimated simply by~\eqref{eq:assumptionA} from Assumption~\ref{as:coefficients}
	\begin{align*}
		-\int_\Om(z-K(v_t))A(v_t)\Delta v_t\dd x&\leq\ve\int_\Om\abs{\Delta v_t}^2\dd x+\frac{\Amax^2}{4\ve}\int_\Om\abs{z-K(v_t)}^2\dd x\\
		&\leq\ve\int_\Om\abs{\Delta v_t}^2\dd x+\frac{\Amax^2}{2\ve}\bigparen{\norm{z}_{\LL^2(\Om)}^2+\Amax^2\norm{v_t}_{\LL^2(\Om)}^2}.
	\end{align*}
	Since $v$ solves \eqref{eq:heatEquation}, the $\LL^2(\Om)$-norms of $v_t$ and $\nabla v_t$ are non-increasing, see \eqref{eq:L2NormDissipationOfHeatEquation} and~\eqref{eq:DirichletEnergyDissipationOfHeatEquation}.
	In summary, there exists a constant $C_4\in(0,\infty)$, depending only on $C_3$, $c_1$ and $\Amax$, such that
	\begin{equation}\label{eq:freeEnergyDerivativeAlongHeatFlowEstimate}
		\ddd{t}\calE[v_t,z]\leq-\frac12\norm{\Delta v_t}_{\LL^2(\Om)}^2+C_4\bigparen{1+\norm{v^0}_{\HH^1(\Om)}^2+\norm{z}_{\LL^2(\Om)}^2}.
	\end{equation}

	\noindent
	\textbf{Step 2:} Continuity $\calE[v_t,z]\to\calE[v^0,z]$ as $t\searrow0$.\\\noindent
	Since $v\in\CC([0,\infty);\HH^1(\Om))$, continuity of the Dirichlet energy and of the coupling term is immediate. For the potential part of the energy, we recall $F=\Fconv+\Fconc+\Faff$. By the bounds~\eqref{eq:potentialDecompositionBounds} it is easy to verify
	\[\lim_{t\searrow0}\int_\Om\bigparen{\Fconc(v_t)+\Faff(v_t)}\dd x=\int_\Om\bigparen{\Fconc(v^0)+\Faff(v^0)}\dd x.\]
	For the convex part, Fatou's lemma and $\Fconv\geq0$ imply on the one hand
	\[\int_\Om\Fconv(v^0)\dd x\leq\liminf_{t\searrow0}\int_\Om\Fconv(v_t)\dd x.\]
	On the other hand, convexity of $\Fconv$ and the fact that $v$ solves~\eqref{eq:heatEquation} imply with Jensen's inequality
	\[\int_\Om\Fconv(v_t)\dd x\leq\int_\Om\Fconv(v^0)\dd x\quad\text{for all }t\in(0,\infty).\]
	Summarising, we have the convergence $\calE[v_t,z]\to\calE[v^0,z]$ for $t\searrow0$.\\

	\noindent
	\textbf{Step 3:} $\HHn^2(\Om)$-regularity and \eqref{eq:freeEnergyLiminfAlongHeatFlowQualitativeEstimate}.\\\noindent
	Now we can apply the mean value theorem and obtain $\xi(s)\in(0,s)$ such that
	\begin{equation}\label{eq:meanValueEstimate}
		\frac{1}{s}\bigparen{\calE[v_s,z]-\calE[v^0,z]}\leq-\frac12\norm{\Delta v_{\xi(s)}}_{\LL^2(\Om)}^2+C_4\bigparen{1+\norm{v^0}_{\HH^1(\Om)}^2+\norm{z}_{\LL^2(\Om)}^2}.
	\end{equation}
	Combining this with \eqref{eq:freeEnergyLiminfAlongHeatFlowBounded}, we get a sequence $(s_n)_{n\in\N}$ with $s_n\searrow0$ and whence $\xi(s_n)\to0$ for $n\to\infty$, such that $\sup_{n\in\N}\norm{\Delta v_{\xi(s_n)}}_{\LL^2(\Om)}<\infty$. Moreover, $\avint{_\Om}v_{\xi(s_n)}\dd x=\avint{_\Om}v^0\dd x=\frakm$ and $\nabla v_{\xi(s_n)}\cdot\n=0$ on $\boundary{\Om}$, so that $(v_{\xi(s_n)})_{n\in\N}$ is uniformly bounded in $\HH^2(\Om)$ by \eqref{eq:H2DeltaEstimate}. Together with $v_{\xi(s_n)}\to v^0$ in $\HH^1(\Om)$ this implies $v^0\in\HH^2(\Om)$, $\nabla v^0\cdot\n=0$ on $\boundary{\Om}$ and $v_{\xi(s_n)}\weakto v^0$ in $\HH^2(\Om)$.
	Estimate~\eqref{eq:freeEnergyLiminfAlongHeatFlowQualitativeEstimate} follows from \eqref{eq:meanValueEstimate}, the weak lower semicontinuity of the norm, and the lower bound \eqref{eq:freeEnergyLowerBound} for $\calE$.\\

	\noindent
	\textbf{Step 4:} Estimate \eqref{eq:freeEnergyLiminfAlongHeatFlowRefinedEstimate}.\\\noindent
	Let $s_n\searrow0$. In view of~\eqref{eq:freeEnergyDerivativeAlongHeatFlow} and the mean value theorem, it suffices to show
	\[I_1(0)+I_2(0)+I_3(0)\leq\limsup_{t\to0}\bigparen{I_1(t)+I_2(t)+I_3(t)}.\]
	Since we have already shown $v^0\in\HH^2(\Om)$, we infer a posteriori $v_t\to v^0$ in $\HH^2(\Om)$ by Proposition~\ref{pr:driftDiffusion} (applied to $V=0$, $\ve=1$) for $t\to0$.
	Thus, $I_1(t)\to I_1(0)$ and $I_3(t)\to I_3(0)$ for $t\to0$.
	To conclude, we need to prove
	\[\int_\Om F''(v^0)\abs{\nabla v^0}^2\dd x\leq\limsup_{t\to0}\int_\Om F''(v_t)\abs{\nabla v_t}^2\dd x.\]
	Fix a subsequence $(t_\ell)_{\ell\in\N}$ such that $t_\ell\to0$, $v_{t_\ell}\to v^0$ and $\nabla v_{t_\ell}\to\nabla v^0$ almost everywhere for $\ell\to\infty$, and recall the decomposition $F''(u)=\Fconv''(u)+\Fconc''(u)$.
	On $\{0<v^0<M\}$ we have $\Fconv''(v_{t_\ell})\abs{\nabla v_{t_\ell}}^2\to\Fconv''(v^0)\abs{\nabla v^0}^2$ almost everywhere by continuity $\Fconv''\in\CC((0,M))$, while, on $\{v^0\in\{0,M\}\}$, we simply use
	\[\liminf_{\ell\to\infty}\Fconv''(v_{t_\ell})\abs{\nabla v_{t_\ell}}^2\geq0=\Fconv''(v^0)\abs{\nabla v^0}^2\]
	almost everywhere. Therefore, Fatou's lemma is applicable, and since the limsup is larger than liminf, we receive
	\[\int_\Om\Fconv''(v^0)\abs{\nabla v^0}^2\dd x\leq\limsup_{\ell\to\infty}\int_\Om\Fconv''(v_{t_\ell})\abs{\nabla v_{t_\ell}}^2\dd x.\]
	For the limit passage in the concave contribution, we will use Vitali's convergence theorem to show $\Fconc''(v_{t_\ell})\abs{\nabla v_{t_\ell}}^2\to\Fconc''(v^0)\abs{\nabla v^0}^2$ in $\LL^1(\Om)$. We first show the local convergence in measure. Precisely as in the proof of the convex part, we have pointwise convergence almost everywhere $\Fconc''(v_{t_\ell})\abs{\nabla v_{t_\ell}}^2\to\Fconc''(v^0)\abs{\nabla v^0}^2$ on $\{0<v^0<M\}$. On $\{v^0\in\{0,M\}\}$, we use that for $\ve,\delta>0$
	\begin{align*}
		\Leb^d\bigparen{\{\abs{\Fconc''(v_{t_\ell})\abs{\nabla v_{t_\ell}}^2}>\ve\}}&\leq\Leb^d\bigparen{\{\abs{m(v_{t_\ell})\Fconc''(v_{t_\ell})}>\delta\}}\\
		&\quad+\Leb^d\bigparen{\bigbraces{\bigabs{\abs{\nabla v_{t_\ell}}^2/m(v_{t_\ell})}>\ve/\delta}}\\
		&\leq\Leb^d\bigparen{\{\abs{m(v_{t_\ell})\Fconc''(v_{t_\ell})}>\delta\}}+\frac{\delta}{\ve}C,
	\end{align*}
	for a constant $C\in(0,\infty)$ depending on $\sup_{\ell\in\N}\norm{v_{t_\ell}}_{\HH^2(\Om)}$, since
	\[\frac{\ve}{\delta}\Leb^d\biggparen{\biggbraces{\Bigabs{\frac{\abs{\nabla v_{t_\ell}}^2}{m(v_{t_\ell})}}>\frac{\ve}{\delta}}}\leq\int_\Om\Bigabs{\frac{\abs{\nabla v_{t_\ell}}^2}{m(v_{t_\ell})}}\dd x\leq C_3\bigparen{\norm{\Delta v_{t_\ell}}_{\LL^2(\Om)}+\norm{\nabla v_{t_\ell}}_{\LL^2(\Om;\R^d)}^2}\]
	by Lemma~\ref{le:nabla2mEstimate}. Since $m(v_{t_\ell})\Fconc''(v_{t_\ell})\to0$ almost everywhere on $\{v^0\in\{0,M\}\}$ and $\ve,\delta$ were arbitrary, we have local convergence in measure.
	Next, we show that $\bigparen{\Fconc''(v_{t_\ell})\abs{\nabla v_{t_\ell}}^2}_{\ell\in\N}$ is uniformly bounded in $\LL^q(\Om)$ space for some $q>1$, thus having uniformly absolutely continuous integrals in $\LL^1(\Om)$. To this end, fix $q\in(1,2]$ such that $\HH^2(\Om)\hookrightarrow\WW^{1,2q}(\Om)$. With condition~\eqref{eq:potentialLowerBound} we find in the case $M<+\infty$
	\begin{align*}
		\int_\Om\bigabs{\Fconc''(v_{t_\ell})\abs{\nabla v_{t_\ell}}^2}^q\dd x&=\int_\Om\Bigabs{\frac{\abs{\nabla v_{t_\ell}}^2}{m(v_{t_\ell})}m(v_{t_\ell})\Fconc''(v_{t_\ell})}^q\dd x\\
		&\leq C_3^q\bigparen{\norm{\Delta v_{t_\ell}}_{\LL^2(\Om)}+\norm{\nabla v_{t_\ell}}_{\LL^{2q}(\Om)}^2}^q\sup_{u\in[0,M]}\abs{m(u)\Fconc''(u)}^q\leq C
	\end{align*}
	for a constant $C\in(0,\infty)$ independent of $\ell$. If $M=+\infty$ we split
	\[\int_\Om\bigabs{\Fconc''(v_{t_\ell})\abs{\nabla v_{t_\ell}}^2}^q\dd x=\int_{\{v_{t_\ell}\leq\frakm\}}\bigabs{\Fconc''(v_{t_\ell})\abs{\nabla v_{t_\ell}}^2}^q\dd x+\int_{\{v_{t_\ell}>\frakm\}}\bigabs{\Fconc''(v_{t_\ell})\abs{\nabla v_{t_\ell}}^2}^q\dd x.\]
	The integral over $\{v_{t_\ell}\leq\frakm\}$ can then be controlled as in the case $M<+\infty$, and the other integral is estimated as follows
	\begin{align*}
		\int_{\{v_{t_\ell}>\frakm\}}\bigabs{\Fconc''(v_{t_\ell})\abs{\nabla v_{t_\ell}}^2}^q\dd x&\leq\norm{\nabla v_{t_\ell}}_{\LL^{2q}(\Om)}^{2q}\sup_{u\in(\frakm,\infty)}\abs{\Fconc''(u)}^q\\
		&\leq c_1\Bigparen{1+\frac{1}{m(\frakm)}}\norm{\nabla v_{t_\ell}}_{\LL^{2q}(\Om)}^{2q},
	\end{align*}
	so that a uniform bound in $\LL^q(\Om)$ is established in any case.
\end{proof}

We now combine the heat-flow estimate with the flow interchange argument to obtain $\HH^2(\Omega)$-regularity for the piecewise constant interpolants.

\begin{lemma}\label{le:aPrioriH2}
	It holds $\wtu_\kappa(t)\in\HHn^2(\Om)$ for every $\kappa>0$, $t\in(0,\infty)$ (in particular for every $u_\kappa^n$, $n\in\N$, since $\wtu_\kappa(n\kappa)=u_\kappa^n$). There exists a constant $C_5\in(0,\infty)$ such that
	\begin{equation}
		\norm{\olu_\kappa}_{\LL^2([0,T];\HH^2(\Om))}+\norm{\ulu_\kappa}_{\LL^2([\kappa,T];\HH^2(\Om))}\leq C_5(T+\kappa+1)\label{eq:aPrioriL2H2olu&ulu}
	\end{equation}
	for every $\kappa>0$ and every $T>0$.
\end{lemma}

\begin{proof}
	Fix $n\in\N_0$, $\theta\in(0,\kappa]$, and set $t:=n\kappa+\theta$. First, we apply the flow interchange Lemma~\ref{le:flowInterchange} to the heat flow $(S_\calU^t)_{t\geq0}$ from~\eqref{eq:defInternalEnergy}. Together with the bounds for $\calU$ in~\eqref{eq:internalEnergyUpperBound} this yields
	\[\liminf_{s\searrow0}\Bigparen{\frac{\calE[S_\calU^s(\wtu_\kappa(t)),\wtz_\kappa(t)]-\calE[\wtu_\kappa(t),\wtz_\kappa(t)]}{s}}\geq\frac{1}{\theta}\bigparen{\calU(\wtu_\kappa(t))-\calU(u_\kappa^n)}\geq-\frac{C_\calU}{\theta}(1+\norm{u_\kappa^n}_{\LL^2(\Om)}^2).\]
	Hence, the condition~\eqref{eq:freeEnergyLiminfAlongHeatFlowBounded} in Lemma~\ref{le:freeEnergyAlongHeatFlow} is fulfilled. It follows $\wtu_\kappa(t)\in\HHn^2(\Om)$ and the estimate~\eqref{eq:freeEnergyLiminfAlongHeatFlowQualitativeEstimate}. Together with the flow interchange estimate~\eqref{eq:flowInterchangeEstimate} this amounts to
	\begin{align}\label{eq:aPrioriL2H2Prelim}
		\norm{\Delta u_\kappa^k}_{\LL^2(\Om)}^2&\leq2C_4(1+\calE(\boldu^0))+\frac{2}{\kappa}\bigparen{\calU(u_\kappa^{k-1})-\calU(u_\kappa^k)}
	\end{align}
	for every $k\in\N$. If $t$ is as before, we obtain by telescoping
	\begin{align*}
		\int_0^t\norm{\Delta\olu_\kappa(t)}_{\LL^2(\Om)}^2\dd t&\leq\kappa\sum_{k=1}^{n+1}\norm{\Delta u_\kappa^k}_{\LL^2(\Om)}^2\\
		&\leq2(n+1)\kappa C_4(1+\calE(\boldu^0))+2\sum_{k=1}^{n+1}\bigparen{\calU(u_\kappa^{k-1})-\calU(u_\kappa^k)}\\
		&\leq2(t+\kappa)C_4(1+\calE(\boldu^0))+2C_\calU(1+C_1^2),
	\end{align*}
	where we used the bounds for $\calU$ from~\eqref{eq:internalEnergyUpperBound} and the a priori estimate~\eqref{eq:aPrioriLIH1u}. Hence, \eqref{eq:aPrioriL2H2olu&ulu} follows for $\olu_\kappa$ from~\eqref{eq:H2DeltaEstimate}, and for $\ulu_\kappa$ since $\norm{\ulu_\kappa}_{\LL^2([\kappa,T];\HH^2(\Om))}=\norm{\olu_\kappa}_{\LL^2([0,T-\kappa];\HH^2(\Om))}$.
\end{proof}

The discrete entropy-dissipation inequality is a direct consequence of the flow interchange argument.

\begin{corollary}[Discrete entropy-dissipation inequality]
	Under Assumption~\ref{as:entropyInequality}~\ref{it:potentialConcavePartVanishing}, we have for every $\kappa>0$, $n\in\N$,
	\begin{align*}
		\calU(\boldu_\kappa^n)+\kappa\int_\Om\Bigbraces{\abs{\Delta u_\kappa^n}^2+F''(u_\kappa^n)\abs{\nabla u_\kappa^n}^2+(z_\kappa^n-K(u_\kappa^n))A(u_\kappa^n)\Delta u_\kappa^n}\dd x\leq\calU(\boldu_\kappa^{n-1}).
	\end{align*}
\end{corollary}

This argument gives no uniform bound for $\wtu_\kappa$, because it does not produce a telescoping sum. Indeed, in view of~\eqref{eq:aPrioriL2H2Prelim}, we only obtain
\begin{align*}
	\int_\theta^\kappa\norm{\Delta\wtu_\kappa(r)}_{\LL^2(\Om)}^2\dd r&\leq2C_3(1+\calE(\boldu^0))\kappa+\int_\theta^\kappa\frac{2}{r}\bigparen{\calU(u_\kappa^k)-\calU(\wtu_\kappa(r))}\dd r\\
	&\leq2C_3(1+\calE(\boldu^0))\kappa+2\calU(u_\kappa^k)\bigparen{\ln\kappa-\ln\theta},
\end{align*}
which blows up for $\theta\searrow0$. Instead, we will use the metric slope estimate~\eqref{eq:slopeEstimate}.

\begin{proposition}\label{pr:aPrioriH2DeGiorgiPrelim}
	There exists a constant $C_6\in(0,\infty)$, only depending on $C_3$ (from Lemma~\ref{le:nabla2mEstimate}), $C_\calE$, $c_1$ and $\Amax$, such that
	\begin{equation}\label{eq:aPrioriH2MetricSlope}
		\frac12\norm{\Delta u}_{\LL^2(\Om)}^2\leq C_6(1+\calE(\boldu))+\abs{\pa\calE}_{\calD_w}^2(\boldu)
	\end{equation}
	for every $\boldu=(u,z)\in\bigparen{\calX_\frakm^M\cap\HHn^2(\Om)}\times\LL^2(\Om)$ and $w\in\calX_\frakm^M$.
\end{proposition}

\begin{proof}
	We let $v_s:=S_\calU^s(u)$, i.e.\ $v$ solves~\eqref{eq:heatEquation} with initial value $u$. If $u\equiv\text{constant}$ in $\Om$ then $\norm{\Delta u}_{\LL^2(\Om)}=0$ and \eqref{eq:aPrioriH2MetricSlope} is trivial; otherwise, it holds $v_s\ne u$ for every $s>0$.

	Noting that $\pa_sv=\divv(\nabla v)$ implies $\calW_m^2(u,v_s)\leq s\int_0^s\int_\Om\frac{\abs{\nabla v_t}^2}{m(v_t)}\dd x\dd t$, we derive the upper bound
	\begin{align*}
		\frac{1}{s^2}\calD_w^2(\boldu,[v_s,z])&\leq\frac{1}{s}\int_0^s\int_\Om\frac{\abs{\nabla v_t}^2}{m(v_t)}\dd x\dd t\leq\frac{1}{s}\int_0^sC_3\bigparen{\norm{\Delta v_t}_{\LL^2(\Om)}+\norm{\nabla v_t}_{\LL^2(\Om;\R^d)}^2}\dd t\\
		&\xrightarrow{s\to0} C_3\bigparen{\norm{\Delta u}_{\LL^2(\Om)}+\norm{\nabla u}_{\LL^2(\Om;\R^d)}^2}.
	\end{align*}
	Here, we used Lemma~\ref{le:nabla2mEstimate} to treat $\frac{\abs{\nabla v_t}^2}{m(v_t)}$, and the convergence $\Delta v_s\to\Delta u$ in $\LL^2(\Om)$ for $s\searrow0$, see Proposition~\ref{pr:driftDiffusion} (applied to $V\equiv0$, $\ve=1$). Moreover, with similar computations as in the proof of Lemma~\ref{le:freeEnergyAlongHeatFlow}, we show
	\[\limsup_{s\searrow0}\frac{1}{s}\bigparen{\calE(\boldu)-\calE[v_s,z]}\geq\frac12\norm{\Delta u}_{\LL^2(\Om)}^2-C_4(1+\calE(\boldu)).\]
	Hence we obtain
	\[\abs{\pa\calE}_{\calD_w}(\boldu)\geq\limsup_{s\searrow0}\frac{[\calE(\boldu)-\calE[v_s,z]]_+}{\calD_w(\boldu,[v_s,z])}\geq\frac{\bigbrackets{\frac12\norm{\Delta u}_{\LL^2(\Om)}^2-C_4(1+\calE(\boldu))}_+}{\bigparen{C_3\bigparen{\norm{\Delta u}_{\LL^2(\Om)}+\norm{\nabla u}_{\LL^2(\Om;\R^d)}^2}}^{1/2}}.\]
	Using Young's inequality
	\begin{align*}
		\frac12\norm{\Delta u}_{\LL^2(\Om)}^2&\leq C_4(1+\calE(\boldu))+\bigparen{C_3\bigparen{\norm{\Delta u}_{\LL^2(\Om)}+\norm{\nabla u}_{\LL^2(\Om;\R^d)}^2}}^{1/2}\abs{\pa\calE}_{\calD_w}(\boldu)\\
		&\leq C_4(1+\calE(\boldu))+\frac14 C_3^2+\frac14\norm{\Delta u}_{\LL^2(\Om)}^2+\frac12 C_3\norm{\nabla u}_{\LL^2(\Om;\R^d)}^2+\frac12\abs{\pa\calE}_{\calD_w}^2(\boldu),
	\end{align*}
	we infer~\eqref{eq:aPrioriH2MetricSlope}.
\end{proof}

\begin{corollary}\label{co:aPrioriH2DeGiorgi}
	There exists a constant $C_7\in(0,\infty)$ such that
	\begin{equation}\label{eq:aPrioriL2H2wtu}
		\norm{\wtu_\kappa}_{\LL^2(0,T;\HH^2(\Om))}\leq C_7(T+1)
	\end{equation}
	for every $\kappa>0$ and $T>0$.
\end{corollary}

\begin{proof}
	Let $T=n\kappa+\theta$ for $n\in\N_0$ and $\theta\in(0,\kappa]$. Applying Proposition~\ref{pr:aPrioriH2DeGiorgiPrelim} to the De\,Giorgi interpolant $\wtu_\kappa(t)$, and making use of the slope estimate~\eqref{eq:slopeEstimate} and the discrete energy-dissipation balance~\eqref{eq:discreteEDB}, we get
	\begin{align*}
		\frac12\int_0^T\norm{\Delta\wtu_\kappa(t)}_{\LL^2(\Om)}^2\dd t&\leq C_6\int_0^T\bigparen{1+\calE(\wtbfu_\kappa(t))}\dd t+\sum_{k=0}^n\int_0^{\kappa}\frac{\calD_{u_\kappa^k}^2(\boldu_\kappa^k,\wtbfu_\kappa(k\kappa+\vartheta))}{2\vartheta^2}\dd\vartheta\\
		&\leq C_6(1+\calE(\boldu^0))T+2\big(\calE(\boldu^0)-\calEmin\big).
	\end{align*}
	The claim follows by \eqref{eq:H2DeltaEstimate}.
\end{proof}

\subsection{Proof of Theorem \ref{th:discreteSolutions}}\label{subsec:discreteSystem}

We now complete the proof of Theorem \ref{th:discreteSolutions} by deriving the discrete evolution equations and establishing the discrete energy-dissipation inequality.
We first start with the equations \eqref{eq:approxEquationU} and \eqref{eq:approxEquationZ} for the PDE and then move on to the EDI.

\begin{lemma}[Discrete relaxation \eqref{eq:approxEquationZ}]\label{le:approxEquationZ}
	For every $\kappa>0$, every $t\in[0,\infty)\setminus\set{n\kappa}{n\in\N}$, it holds in $\LL^2(\Om)$
	\[\pa_t\whz_\kappa(t)=-\frac{1}{\tau(\ulu_\kappa(t))}\bigparen{\olz_\kappa(t)-K(\olu_\kappa(t))}.\]
\end{lemma}

\begin{proof}
	We just compute the Euler--Lagrange equations of~\eqref{eq:TIM} with respect to $z$. Fix $\kappa>0$, $n\in\N_0$, and let $p\in\LL^2(\Om)$ be arbitrary. For $h\in(0,\infty)$ we compute
	\begin{align*}
		&\frac{\Psi_{\kappa,\boldu_\kappa^n}[u_\kappa^{n+1},z_\kappa^{n+1}+hp]-\Psi_{\kappa,\boldu_\kappa^n}[u_\kappa^{n+1},z_\kappa^{n+1}]}{h}\\
		&\qquad=\frac{1}{\kappa}\bigparen{\tau(u_\kappa^n)(z_\kappa^{n+1}-z_\kappa^n),p}_{\LL^2(\Om)}+\int_\Om\bigparen{z_\kappa^{n+1}-K(u_\kappa^{n+1})}p\dd x+O(h).
	\end{align*}
	Thus, $h\mapsto\Psi_{\kappa,\boldu_\kappa^n}[u_\kappa^{n+1},z_\kappa^{n+1}+hp]$ is differentiable at $h=0$, and since $\boldu_\kappa^{n+1}$ is a minimiser,
	\[0=\ddd{h}\bigg\vert_{h=0}\Psi_{\kappa,\boldu_\kappa^n}[u_\kappa^{n+1},z_\kappa^{n+1}+hp]=\int_\Om\Bigparen{\tau(u_\kappa^n)\frac{z_\kappa^{n+1}-z_\kappa^n}{\kappa}+\bigparen{z_\kappa^{n+1}-K(u_\kappa^{n+1})}}p\dd x.\]
	Since $p\in\LL^2(\Om)$ was arbitrary, the statement follows directly from the definition of the interpolants $\whz_\kappa$, $\olz_\kappa$, $\olu_\kappa$, $\ulu_\kappa$ in~\eqref{eq:piecewiseInterpolants}.
\end{proof}

As already outlined in Section~\ref{sec:setupAndStrategy}, the derivation of~\eqref{eq:approxEquationU} relies on the flow interchange technique and the regularised potential energy~\eqref{eq:defPotentialEnergyReg}. To this end, we need to estimate the dissipation
\[\liminf_{s\searrow0}\frac{1}{s}\bigparen{\calE[v_s,z]-\calE[v^0,z]},\qquad v_s=S_{\calV_\ve}^s(v^0)\]
suitably. In fact, we establish this for a more general evolution as we will need it later for the definitions of $\olj_\kappa$ and $\wtj_\kappa$. To this end, recall the definition of $\boldN$ given in~\eqref{eq:defBoldN}.

\begin{lemma}\label{le:freeEnergyAlongPotentialFlow}
	Let $\boldzeta\in\CC^1(\varcl{\Om};\R^d)$ with $\boldzeta\cdot\n=0$ on $\boundary{\Om}$, $(v^0,z)\in(\calX_\frakm^M\cap\HHn^2(\Om))\times\LL^2(\Om)$, and $\ve>0$. Moreover, assume $s\mapsto v_s$ solves the nonlinear drift-diffusion equation~\eqref{eq:viscousPotentialEquation} with $v_0=v^0$, and fix $p\in\LL^2(\Om)$. Then
	\begin{multline}\label{eq:freeEnergyLiminfAlongPotentialFlowQualitativeEstimate}
		-\liminf_{s\searrow0}\frac{1}{s}\bigparen{\calE[v_s,z+sp]-\calE[v^0,z]}\\
		\geq-\boldN[(v^0,z),\boldzeta]-(p,z-K(v^0))_{\LL^2(\Om)}-\ve C_8(1+\calE[v^0,z])
	\end{multline}
	for some constant $C_8\in(0,\infty)$ depending only on $C_3$ (from Lemma~\ref{le:nabla2mEstimate}), $C_\calE$, $c_1$ and $\Amax$.
\end{lemma}

\begin{proof}
	Similarly to Lemma~\ref{le:freeEnergyAlongHeatFlow}, we want to use absolute continuity of $t\mapsto\calE[v_t,z]$. To be able to differentiate, we need to avoid evaluations of $F$ at its singularities. Since the strong maximum principle does not apply to~\eqref{eq:viscousPotentialEquation}, we must regularise $F$.\\

	\noindent
	\textbf{Step 1:} Regularisation of $F$.\\\noindent
	Introducing $\alpha_\delta(u):=\delta+\frac{M-2\delta}{M}u$ for $0<\delta<M/2$, with convention $\frac{+\infty-2\delta}{+\infty}:=1$, we put $F_\delta(u):=F(\alpha_\delta(u))$. Observe that $F_\delta\in\CC^2(\rangeM)$ and $F_\delta(u)\to F(u)$ for every $u\in\rangeM$. Also note that the regularised potential satisfies
	\begin{equation}\label{eq:potentialLowerBoundReg}
		m(u)F_\delta''(u)\geq-c_1(2+m(u))
	\end{equation}
	for all small $\delta>0$. This is of course trivial if $F_\delta''(u)\geq0$, and for $u$ with $F_\delta''(u)<0$ we have
	\begin{align}\label{eq:potentialLowerBoundRegPrelim}
		m(u)F_\delta''(u)&=m(\alpha_\delta(u))F_\delta''(u)+\bigparen{m(u)-m(\alpha_\delta(u)}F_\delta''(u)\notag\\
		&\geq\bigparen{\tfrac{M-2\delta}{M}}^2\Bigparen{-c_1\bigparen{1+m(\alpha_\delta(u))}+\bigparen{m(u)-m(\alpha_\delta(u))}F''(\alpha_\delta(u))}\notag\\
		&\geq-c_1\bigparen{\tfrac{M-2\delta}{M}}^2\bigparen{1+\norm{m'}_\infty\delta\bigabs{1-\tfrac{2u}{M}}+m(u)}\notag\\
		&\qquad+\bigparen{\tfrac{M-2\delta}{M}}^2\bigparen{m(u)-m(\alpha_\delta(u))}F''(\alpha_\delta(u)),
	\end{align}
	with convention $\frac{2u}{+\infty}:=0$, where we used \eqref{eq:potentialLowerBound} and Lipschitz continuity of $m$~\eqref{eq:mobilityLipschitzConditions}. If $M=+\infty$ then $m$ is even non-decreasing, and hence \eqref{eq:potentialLowerBoundRegPrelim} implies $m(u)F_\delta''(u)\geq-c_1(1+\norm{m'}_\infty\delta+m(u))$. If $M<+\infty$, then we find $s_0\in(0,M)$ such that $\norm{m}_\infty=m(s_0)$, and $m\vert_{[0,s_0]}$ is monotonically increasing, $m\vert_{[s_0,M]}$ is monotonically decreasing. Thus, \eqref{eq:potentialLowerBoundRegPrelim} can be estimated by
	\[m(u)F_\delta''(u)\geq-c_1\bigparen{1+\norm{m'}_\infty\delta+m(u)}-\norm{m'}_\infty\delta\norm{F''}_{\CC([s_0\wedge\frac{M}{2},s_0\vee\frac{M}{2}])}.\]
	Furthermore, with $P_\delta(u)=\int_\frakm^um(r)F_\delta''(r)\dd r$, we have in the case $M=+\infty$
	\begin{equation}\label{eq:asymptoticBehaviorPReg}
		\lim_{u\to\infty}\sup_{\delta\in[0,1]}\frac{P_\delta(u)}{u^\omega+\abs{F_\delta(u)}}=0.
	\end{equation}
	Indeed, this can be deduced from~\eqref{eq:asymptoticBehaviorP} and~\eqref{eq:pressureUpperBoundReg} below, which follows from Lipschitz continuity of $m$ from~\eqref{eq:mobilityLipschitzConditions} and $\abs{F''(r)}\leq F''(r)+2c_1(1+\frac{1}{m(r)})$ from~\eqref{eq:potentialLowerBound},
	\begin{align}\label{eq:pressureUpperBoundReg}
		P_\delta(u)&=\int_\frakm^u\bigparen{m(r)-m(r+\delta)+m(r+\delta)}F''(r+\delta)\dd r\notag\\
		&\leq\int_\frakm^{u+\delta}\bigparen{\norm{m'}_\infty\delta+m(r)}\abs{F''(r)}\dd r\notag\\
		&\leq\Bigparen{\frac{\norm{m'}_\infty\delta}{m(\frakm)}+1}\int_\frakm^{u+\delta}m(r)\Bigparen{F''(r)+2c_1\Bigparen{1+\frac{1}{m(\frakm)}}}\dd r\notag\\
		&\leq\Bigparen{\frac{\norm{m'}_\infty\delta}{m(\frakm)}+1}\Bigparen{P(u+\delta)+2c_1\Bigparen{1+\frac{1}{m(\frakm)}}\norm{m'}_\infty(u+\delta)^2}.
	\end{align}

	\noindent
	\textbf{Step 2:} Preliminary version of \eqref{eq:freeEnergyLiminfAlongPotentialFlowQualitativeEstimate}.\\\noindent
	Now let $\calE_\delta$ denote the energy obtained from $\calE$ by replacing $F$ with $F_\delta$. Abbreviate $z_s:=z+sp$. Exploiting the regularity of $v$ from Proposition~\ref{pr:driftDiffusion}, we get for every $s\in(0,\infty)$
	\begin{align*}
		\calE_\delta[v_s,z_s]-\calE_\delta[v^0,z]&\leq\ve\int_0^s\int_\Om\Bigparen{-\abs{\Delta v_t}^2+F_\delta'(v_t)\Delta v_t-A(v_t)(z_t-K(v_t))\Delta v_t}\dd x\dd t\\
		&\quad+\int_0^s\int_\Om\divv(m(v_t)\boldzeta)\bigparen{-\Delta v_t+F_\delta'(v_t)}\dd x\dd t\\
		&\quad-\int_0^s\int_\Om\divv(m(v_t)\boldzeta)A(v_t)(z_t-K(v_t))-p(z_t-K(v_t))\dd x\dd t\\
		&=\ve I(s)+\int_0^s\bigparen{\boldN_\delta[(v_t,z_t),\boldzeta]+(p,z_t-K(v_t))_{\LL^2(\Om)}}\dd t,
	\end{align*}
	where $\boldN_\delta$ is obtained from $\boldN$ by replacing $P$ with $P_\delta$.
	To obtain the last line, we integrated by parts twice with the boundary terms vanishing due to $\nabla v_t\cdot\n=\boldzeta\cdot\n=0$ on $\boundary{\Om}$,
	\[\int_\Om\divv(m(v_t)\boldzeta)F_\delta'(v_t)\dd x=-\int_\Om m(v_t)\boldzeta\cdot F_\delta''(v_t)\nabla v_t\dd x=\int_\Om P_\delta(v_t)\divv\boldzeta\dd x.\]
	Taking into account \eqref{eq:potentialLowerBoundReg}, we can treat the error integral $I(t)$ in the same way as in Step 1 of the proof of Lemma~\ref{le:freeEnergyAlongHeatFlow}. The only difference is that we cannot use that $s\mapsto\norm{v_s}_{\LL^2(\Om)}^2,\norm{\nabla v_s}_{\LL^2(\Om;\R^d)}^2$ are non-increasing. Therefore, we end up with
	\begin{align*}
		\frac{\calE_\delta[v_s,z_s]-\calE_\delta[v^0,z]}{s}\leq\frac{1}{s}\int_0^s\Bigparen{&\boldN_\delta[(v_t,z_t),\boldzeta]+(p,z_t-K(v_t))_{\LL^2(\Om)}\\
			&+\ve C_{8,1}\bigparen{\norm{\nabla v_t}_{\LL^2(\Om;\R^d)}^2+\norm{z_t-K(v_t)}_{\LL^2(\Om)}^2+1}}\dd t
	\end{align*}
	for a constant $C_{8,1}\in(0,\infty)$ depending only on $C_3$, $c_1$ and $\Amax$.\\

	\noindent
	\textbf{Step 3:} Limit passages $\delta\searrow0$, $s\searrow0$.\\\noindent
	First, we pass to the limit $\delta\searrow0$ for $s>0$ fixed. The pointwise convergence $F_\delta(v_s)\to F(v_s)$ is upgraded to $\LL^1(\Om)$-convergence thanks to $v_s\in\HH^2(\Om)$ and the growth condition~\eqref{eq:potentialUpperBound}. Moreover, we have $P_\delta(v_s)\to P(v_s)$ in $\LL^1(\Om)$ due to Vitali convergence theorem. Indeed, pointwise convergence is clear, so it remains to show uniform integrability. By~\eqref{eq:asymptoticBehaviorPReg} and \eqref{eq:potentialUpperBound}, the family $(P_\delta(v_s))_{\delta\in[0,1]}$ is uniformly bounded in $\LL^1(\Omega)$. Moreover, we find for every $\ve>0$ a constant $C_\ve\in(0,\infty)$ such that $\abs{P_\delta(u)}\leq\ve(u^\omega+\abs{F_\delta(u)})$ for every $u\in[C_\ve,+\infty)$, $\delta\in[0,1]$. Hence, if $K\subset\Om$ fulfils $\Leb^d(K)\leq\frac{\ve}{\norm{P}_{\CC([0,C_\ve+1])}}$, then
	\begin{align*}
		\int_K\abs{P_\delta(v_s)}\dd x&=\int_{\{v_s<C_\ve\}\cap K}\abs{P_\delta(v_s)}\dd x+\int_{\{v_s\geq C_\ve\}\cap K}\abs{P_\delta(v_s)}\dd x\\
		&\leq\ve+\ve\bigparen{\norm{v_s}_{\LL^\omega(\Om)}^\omega+\norm{F_\delta(v_s)}_{\LL^1(\Om)}}.
	\end{align*}
	Thus, \eqref{eq:potentialUpperBound} guarantees that $(P_\delta(v_s))_{\delta\in[0,1]}$ has uniformly absolutely continuous integrals. Therefore, the passage $\delta\searrow0$ is justified.

	In order to perform the limit $s\searrow0$, we use the continuity properties of $v$ from Proposition~\ref{pr:driftDiffusion} and the energy bound \eqref{eq:freeEnergyLowerBound} so that we conclude with $C_8:=8C_{8,1}(C_\calE+1)$.
\end{proof}

We now apply this dissipation estimate to the regularised potential flow to derive the approximate weak equation for $\olbfu_\kappa$.

\begin{lemma}[Approximate equation for $u_\kappa$]
	Let $V\in\CC^2(\varcl{\Om})\cap\HHn^2(\Om)$ be a spatial test function and $\psi\in\CC_c^1(0,\infty)$ a temporal test function be given. Fix $\kappa>0$ and $\ve>0$. Then
	\begin{equation}\label{eq:approxEquationUprelim}
		\biggabs{\int_0^\infty\psi'(t)\calV(\olu_\kappa(t))\dd t+\int_0^\infty\underline{\psi}_\kappa(t)\boldN[\olbfu_\kappa(t),\nabla V]\dd t}\leq C_9\Bigparen{\frac{\kappa}{\ve}+\ve+\kappa},
	\end{equation}
	where $C_9\in(0,\infty)$ is a constant only depending on $C_3$ (from Lemma~\ref{le:nabla2mEstimate}), $C_\calU$, $C_\calE$, $\calEmin$, $\calE(\boldu^0)$, $c_1$, $\Amax$, $\psi$, and $V$. Here, $\underline{\psi}_\kappa$ is the piecewise constant function defined as $\underline{\psi}_\kappa(t)=\psi\bigparen{\kappa\lfloor\frac{t}{\kappa}\rfloor}$. In particular, \eqref{eq:approxEquationU} follows from \eqref{eq:approxEquationUprelim} with the choice $\ve:=\sqrt{\kappa}$.
\end{lemma}

\begin{proof}
	The idea is to apply the flow interchange estimate~\eqref{eq:flowInterchangeEstimate} to the regularised drift flow, generated by the potential $\calV_\ve$ introduced in~\eqref{eq:defPotentialEnergyReg}. First we note that it suffices to check~\eqref{eq:approxEquationUprelim} for non-negative $\psi$. Indeed, using $-\calV_V=\calV_{-V}$ and $\boldN[\boldu,-\nabla V]=-\boldN[\boldu,\nabla V]$, the case of a general $\psi$ follows by decomposing $\psi$ into its positive and negative parts, and replacing $V$ with $-V$ for the negative part.

	Rearranging the first integral in \eqref{eq:approxEquationUprelim} yields
	\begin{align}\label{eq:approxEquationUprelimLHS}
		-\int_0^\infty\psi'(t)\calV_\ve(\olu_\kappa(t))\dd t&=-\sum_{n=1}^\infty\int_{(n-1)\kappa}^{n\kappa}\psi'(t)\calV_\ve(u_\kappa^n)\dd t\notag\\
		&=\sum_{n=1}^\infty\bigparen{\psi((n-1)\kappa)-\psi(n\kappa)}\calV_\ve(u_\kappa^n)\notag\\
		&=\sum_{n=0}^\infty\psi(n\kappa)\bigparen{\calV_\ve(u_\kappa^{n+1})-\calV_\ve(u_\kappa^n)}.
	\end{align}
	Note that the integral and all sums are finite since $\psi$ has compact support. We also used $\psi(0)=0$ in the last line. Applying the flow interchange estimate~\eqref{eq:flowInterchangeEstimate} to $\calF=\calV_\ve$ and combining it with Lemma~\ref{le:freeEnergyAlongPotentialFlow} (here with $p=0$), which is applicable thanks to Lemma~\ref{le:aPrioriH2}, yields
	\[\calV_\ve(u_\kappa^{n+1})-\calV_\ve(u_\kappa^n)\leq\kappa\boldN[\boldu_\kappa^{n+1},\nabla V]+\kappa\ve C_8(1+\calE(\boldu_\kappa^{n+1}))-\frac{\lambda_\ve}{2}\calW_m^2(u_\kappa^{n+1},u_\kappa^n)\]
	for every $\kappa>0$, $n\in\N_0$. Inserting this into~\eqref{eq:approxEquationUprelimLHS} and using $\psi\geq0$, we end up with
	\begin{align}\label{eq:approxEquationUprelim2}
		-\int_0^\infty\psi'(t)\calV_\ve(\olu_\kappa(t))\dd t&\leq\sum_{n=0}^\infty\psi(n\kappa)\Bigparen{\kappa\boldN[\boldu_\kappa^{n+1},\nabla V]-\frac{\lambda_\ve}{2}\calW_m^2(u_\kappa^{n+1},u_\kappa^n)}\notag\\
		&\qquad+\sum_{n=0}^\infty\psi(n\kappa)\kappa\ve C_8(1+\calE(\boldu_\kappa^{n+1}))\notag\\
		&\leq\int_0^\infty\underline{\psi}_\kappa(t)\boldN[\olbfu_\kappa(t),\nabla V]\dd t+|\lambda_\ve|\kappa\norm{\psi}_\infty(\calE(\boldu^0)-\calEmin)\notag\\
		&\qquad+\ve C_8(1+\calE(\boldu^0))\int_0^\infty\underline{\psi}_\kappa(t)\dd t.
	\end{align}
	In the last line we used the telescope estimate~\eqref{eq:TIMTelescopeEstimate}.
	Since $\calV_\ve=\calV+\ve\calU$, we still need to estimate the contribution of $\calU$ in the left-hand side of~\eqref{eq:approxEquationUprelim2}. Relying on \eqref{eq:internalEnergyUpperBound} and \eqref{eq:aPrioriLIH1u}, we have
	\begin{align*}
		\int_0^\infty\psi'(t)\calU(\olu_\kappa(t))\dd t&\leq C_\calU\int_0^\infty\abs{\psi'(t)}\bigparen{1+\norm{\olu_\kappa(t)}_{\LL^2(\Om)}^2}\dd t\\
		&\leq C_\calU\bigparen{1+C_1^2}\int_0^\infty\abs{\psi'(t)}\dd t.
	\end{align*}
	Inserting this into~\eqref{eq:approxEquationUprelim2}, and noting that $\lambda_\ve\geq-C(1+\frac{1}{\ve})$ for a constant $C>0$ depending on $m$ and $V$, we receive
	\begin{equation}
		-\int_0^\infty\psi'(t)\calV(\olu_\kappa(t))\dd t\leq\int_0^\infty\underline{\psi}_\kappa(t)\boldN[\olbfu_\kappa(t),\nabla V]\dd t+C_9\Bigparen{\frac{\kappa}{\ve}+\ve+\kappa}.
	\end{equation}
	Replacing $V$ by $-V$ yields precisely \eqref{eq:approxEquationUprelim}.
\end{proof}

We next establish the flux estimates needed to prove the discrete energy-dissipation inequality \eqref{eq:approxFreeEnergyEDI} and the constitutive relations \eqref{eq:approxFluxConditions}.

\begin{proposition}[Operator bound for constitutive flux]\label{pr:metricSlopeEstimate}
	Let $\boldu=(u,z)\in(\HHn^2(\Om)\cap\calX_\frakm^M)\times\LL^2(\Om)$, $w\in\calX_\frakm^M$, and suppose that $\abs{\pa\calE}_{\calD_w}(\boldu)<\infty$. Then, for every $\boldzeta\in\CC^1(\Om;\R^d)$ with $\boldzeta\cdot\n=0$ and $p\in\LL^2(\Om)$ it holds
	\begin{equation}\label{eq:operatorBoundFlux}
		-\boldN[\boldu,\boldzeta]-(p,z-K(u))_{\LL^2(\Om)}\leq\Bigparen{\int_\Om\bigparen{m(u)\abs{\boldzeta}^2+\tau(w)p^2}\dd x}^{1/2}\abs{\pa\calE}_{\calD_w}(\boldu).
	\end{equation}
	Consequently, there exists a unique continuous linear functional $L_{\boldu,w}$ on the separable Hilbert space $Y:=\LL^2\bigparen{\Om,m(u)\Leb^d\vert_\Om;\R^d}\times\LL^2(\Om,\tau(w)\Leb^d\vert_\Om)$ such that
	\begin{equation}\label{eq:constitutiveFluxLinearFunctional}
		L_{\boldu,w}(\boldzeta,p)=-\boldN[\boldu,\boldzeta]-(p,z-K(u))_{\LL^2(\Om)},
	\end{equation}
	for every $(\boldzeta,p)\in\set{\tilde{\boldzeta}\in\CC^1(\Om;\R^d)}{\tilde{\boldzeta}\cdot\n=0}\times\LL^2(\Om)$, identified with its equivalence class in $Y$. Furthermore,
	\begin{align}\label{eq:L-norm}
		\norm{L_{\boldu,w}}_{Y^*}\leq\abs{\pa\calE}_{\calD_{w}}(\boldu). 
	\end{align}
\end{proposition}

\begin{proof}
 To show~\eqref{eq:operatorBoundFlux}, we distinguish two cases.\\

	\noindent
	\textbf{Case 1:} $m(u)\boldzeta\not\equiv0$ or $p\not\equiv0$.\\\noindent
	We may assume that the right-hand side in~\eqref{eq:operatorBoundFlux} is finite, otherwise there is nothing to do. Abbreviate $s\mapsto z_s:=z+sp$ and let $s\mapsto v_{\ve,s}$ be the solution to~\eqref{eq:viscousPotentialEquation} with initial value $v_{\ve,0}=u$ for $\ve>0$. We emphasise that $m(u)\boldzeta\not\equiv0$ would imply $v_{\ve,s}\ne u$ for every $s\in(0,\infty)$ and small enough $\ve$, so that $(v_{\ve,s},z_s)\ne\boldu$ for every $s\in(0,\infty)$ is guaranteed.

	Similarly as in Proposition~\ref{pr:aPrioriH2DeGiorgiPrelim} we use the definition of $\calW_m$ and \eqref{eq:viscousPotentialEquation} $\pa_sv_\ve=\divv(m(v_\ve)\boldzeta+\ve\nabla v_\ve)$ to obtain $\calW_m^2(u,v_{\ve,s})\leq s\int_0^s\int_\Om\frac{\abs{m(v_\ve)\boldzeta+\ve\nabla v_\ve}^2}{m(v_\ve)}\dd x\dd t$, so that
	\begin{align*}
		\frac{1}{s^2}\calD_w^2(\boldu,[v_{\ve,s},z_s])&\leq\frac{1}{s}\int_0^s\int_\Om\frac{\abs{m(v_{\ve,t})\boldzeta+\ve\nabla v_{\ve,t}}^2}{m(v_{\ve,t})}\dd x\dd t+\int_\Om\tau(w)p^2\dd x\\
		&\leq\frac{1}{s}\int_0^s\int_\Om\bigparen{m(v_{\ve,t})\abs{\boldzeta}^2+2\ve\abs{\boldzeta}\abs{\nabla v_{\ve,t}}}\dd x\dd t+\int_\Om\tau(w)p^2\dd x\\
		&\quad+\frac{\ve^2C_3}{s}\int_0^s\bigparen{\norm{\Delta v_{\ve,t}}_{\LL^2(\Om)}+\norm{\nabla v_{\ve,t}}_{\LL^2(\Om;\R^d)}^2}\dd x\dd t\\
		&\xrightarrow{s\searrow0}\int_\Om\bigparen{m(u)\abs{\boldzeta}^2+\tau(w)p^2}\dd x\\
		&\quad+2\ve\bignorm{\abs{\boldzeta}\abs{\nabla u}}_{\LL^1(\Om)}+\ve^2C_3\bigparen{\norm{\Delta u}_{\LL^2(\Om)}+\norm{\nabla u}_{\LL^2(\Om;\R^d)}^2}.
	\end{align*}
	Again, we relied on \eqref{eq:nabla2mEstimate} to estimate $\frac{\abs{\nabla v_t}^2}{m(v_t)}$, and on Proposition~\ref{pr:driftDiffusion} to justify the limit passage $s\searrow0$. Hence, invoking Lemma~\ref{le:freeEnergyAlongPotentialFlow}, we get
	\begin{align*}
		\abs{\pa\calE}_{\calD_w}(\boldu)&=\limsup_{\boldv\to\boldu}\frac{[\calE(\boldu)-\calE(\boldv)]_+}{\calD_w(\boldu,\boldv)}\geq\sup_{\ve>0}\limsup_{s\searrow0}\frac{[\calE(\boldu)-\calE[v_{\ve,s}(u),z_s]]_+}{\calD_w(\boldu,[v_{\ve,s}(u),z_s])}\\
		&\geq\sup_{\ve>0}\frac{\bigbrackets{-\boldN[\boldu,\boldzeta]-(p,z-K(u))_{\LL^2(\Om)}-\ve C_8(1+\calE(\boldu))}_+}{\bigparen{\int_\Om(m(u)\abs{\boldzeta}^2+\tau(w)p^2)\dd x+2\ve\norm{\abs{\boldzeta}\abs{\nabla u}}_{\LL^1(\Om)}+\ve^2C(u)}^{1/2}}\\
		&=\frac{\bigbrackets{-\boldN[\boldu,\boldzeta]-(p,z-K(u))_{\LL^2(\Om)}}_+}{\bigparen{\int_\Om(m(u)\abs{\boldzeta}^2+\tau(w)p^2)\dd x}^{1/2}}.
	\end{align*}

	\noindent
	\textbf{Case 2:} $m(u)\boldzeta=0$ and $p=0$.\\\noindent
	Since the right-hand side in~\eqref{eq:operatorBoundFlux} is zero in this case, we need to show $\boldN[\boldu,\boldzeta]=0$. Recalling the definition of $\boldN$ in~\eqref{eq:defBoldN}, we observe that $\int_\Om P(u)\divv(\boldeta)\dd x$ is the only term requiring further consideration. To handle this, we use the regularisations $F_\delta$ and $P_\delta$ introduced in Lemma~\ref{le:freeEnergyAlongPotentialFlow}. Integrating by parts for fixed $\delta>0$ and then passing to the limit $\delta\searrow0$ with similar arguments as in Step~3 of Lemma~\ref{le:freeEnergyAlongPotentialFlow} yields
	\begin{align*}
		\int_\Om P(u)\divv(\boldeta)\dd x&=\lim_{\delta\searrow0}\int_\Om P_\delta(u)\divv(\boldeta)\dd x\\
		&=\lim_{\delta\searrow0}-\int_\Om m(u)\boldeta\cdot F_\delta''(u)\nabla u\dd x=0,
	\end{align*}
	which concludes the proof of the bound~\eqref{eq:operatorBoundFlux}.\smallskip

	The asserted continuous extension of $L_{\boldu,w}$ follows from~\eqref{eq:operatorBoundFlux} combined with $\tau(w)\sim1$ and the density of admissible test functions $(\boldzeta,p)$ in $Y$.
\end{proof}

Proposition~\ref{pr:metricSlopeEstimate} shows that the variation of the energy defines a bounded functional on the relevant weighted Hilbert space, with norm controlled by the metric slope. We now apply the Riesz representation theorem to construct fluxes for the piecewise constant and De\,Giorgi interpolants and obtain corresponding dissipation bounds.

\begin{proposition}[Flux interpolants $\olj_\kappa$, $\wtj_\kappa$]\label{pr:existencePiecewiseFlux}
	For every $\kappa>0$, there exist fluxes $\olj_\kappa,\wtj_\kappa\in\LL^2([0,\infty);\LL^{4/3}(\Om;\R^d))$ satisfying, for almost every $t\in[0,\infty)$,
	\begin{equation}\label{eq:discreteFluxCharacterisation}
		\int_\Om\olj_\kappa\cdot\boldzeta\dd x={\boldN[\olbfu_\kappa,\boldzeta]},\qquad \int_\Om\wtj_\kappa\cdot\boldzeta\dd x={\boldN[\wtbfu_\kappa,\boldzeta]}
	\end{equation}
	for all $\boldzeta\in\CC^1(\varcl{\Om};\R^d)$ with $\boldzeta\cdot\n=0$ on $\boundary{\Om}$, and
	\begin{align}
		\int_\Om\Bigparen{\frac{\abs{\olj_\kappa}^2}{m(\olu_\kappa)}+\frac{1}{\tau(\ulu_\kappa)}\abs{\olz_\kappa-K(\olu_\kappa)}^2}\dd x&\leq\abs{\pa\calE}_{\calD_{\ulu_\kappa}}^2(\olbfu_\kappa),\label{eq:PcdiscreteFluxSlopeEstimate}\\
		\int_\Om\Bigparen{\frac{\abs{\wtj_\kappa}^2}{m(\wtu_\kappa)}+\frac{1}{\tau(\ulu_\kappa)}\abs{\wtz_\kappa-K(\wtu_\kappa)}^2}\dd x&\leq\abs{\pa\calE}_{\calD_{\ulu_\kappa}}^2(\wtbfu_\kappa).\label{eq:discreteFluxSlopeEstimate}
	\end{align}
	Moreover, it holds the following a priori estimate
	\begin{equation}\label{eq:aPrioriL2L43boldj}
		\norm{\olj_\kappa}_{\LL^2([0,\infty);\LL^{4/3}(\Om;\R^d))}^2+\norm{\wtj_\kappa}_{\LL^2([0,\infty);\LL^{4/3}(\Om;\R^d))}^2\leq2C_2\norm{m'}_\infty(\calE(\boldu^0)-\calEmin)
	\end{equation}
	with $C_2$ from Lemma~\ref{le:aPrioriEasy}.
\end{proposition}

\begin{proof}
	Fix $n\in\N$. We want to apply Proposition~\ref{pr:metricSlopeEstimate} to $\boldu=\boldu_\kappa^n,\wtbfu_\kappa(t)$ and $w=u_\kappa^{n-1}$ with $t\in[n\kappa,(n+1)\kappa)$. Let $\mathbb{I}_R:Y\to Y^*$ denote the Riesz isomorphism. Since the discrete EDB~\eqref{eq:discreteEDB} implies $\abs{\pa\calE}_{\calD_w}(\boldu_\kappa^n),\abs{\pa\calE}_{\calD_w}(\wtbfu_\kappa(t))<\infty$ for almost every $t$, we can define
	\[(J^n,\varrho^n):=\mathbb{I}_R^{-1}(L_{\boldu_\kappa^n,w}),\qquad(\widetilde{J}^n,\widetilde{\varrho}^n)_t:=\mathbb{I}_R^{-1}(L_{\wtbfu_\kappa(t),w}).\]
	Note that $t\mapsto(\widetilde{J}^n,\widetilde{\varrho}^n)_t$ is Bochner measurable since $t\mapsto L_{\wtbfu_\kappa(t),w}$ is Bochner measurable and $\mathbb{I}_R$ is continuous. Defining $\boldj_\kappa^n:=-m(u_\kappa^n)J^n\in\LL^2(\Om,\frac{1}{m(u_\kappa^n)}\Leb^d\vert_\Om;\R^d)$ and recalling the definition of $L_{\boldu_\kappa^n,w}$ in Proposition~\ref{pr:metricSlopeEstimate}, we see that $\boldj_\kappa^n$ satisfies the first identity in~\eqref{eq:discreteFluxCharacterisation} on $(n\kappa,(n+1)\kappa)$. Moreover, comparing the expressions for arbitrary $p$ in~\eqref{eq:constitutiveFluxLinearFunctional}, we find $\varrho^n=-\frac{1}{\tau(u_\kappa^{n-1})}(z_\kappa^n-K(u_\kappa^n))$. Furthermore,
	\begin{align*}
		\int_\Om\Bigparen{\frac{\abs{\boldj_\kappa^n}^2}{m(u_\kappa^n)}+\frac{1}{\tau(u_\kappa^{n-1})}\abs{z_\kappa^n-K(u_\kappa^n)}^2}\dd x&=\norm{(J^n,\varrho^n)}_Y^2\\
		&=\norm{L_{\boldu_\kappa^n,w}}_{Y^*}^2\leq\abs{\pa\calE}_{\calD_w}^2(\boldu_\kappa^n).
	\end{align*}
	Similarly, we show that $\wtj_\kappa^n:=-m(\wtu_\kappa)\widetilde{J}^n$ satisfies the second identity in~\eqref{eq:discreteFluxCharacterisation} on $(n\kappa,(n{+}1)\kappa)$, that $\widetilde{\varrho}^n=-\frac{1}{\tau(u_\kappa^{n-1})}(\wtz_\kappa-K(\wtu_\kappa))$, and
	\[\int_\Om\Bigparen{\frac{\abs{\wtj_\kappa^n}^2}{m(\wtu_\kappa)}+\frac{1}{\tau(u_\kappa^{n-1})}\abs{\wtz_\kappa-K(\wtu_\kappa)}^2}\dd x\leq\abs{\pa\calE}_{\calD_w}^2(\wtbfu_\kappa).\]
	Thus, if we define $\olj_\kappa(t):=\boldj_\kappa^{\lceil t/\kappa\rceil}$ and $\wtj_\kappa(t):=\wtj_\kappa^{\lceil t/\kappa\rceil}(t)$, we immediately infer \eqref{eq:discreteFluxCharacterisation}, \eqref{eq:PcdiscreteFluxSlopeEstimate} and \eqref{eq:discreteFluxSlopeEstimate}. It remains to see~\eqref{eq:aPrioriL2L43boldj}. H\"older's inequality yields for almost every $t\in[0,\infty)$
	\[\norm{\olj_\kappa(t)}_{\LL^{4/3}(\Om)}^2=\Bignorm{\sqrt{m(\olu_\kappa(t))}\frac{\olj_\kappa(t)}{\sqrt{m(\olu_\kappa(t))}}}_{\LL^{4/3}(\Om)}^2\leq\Bignorm{m(\olu_\kappa(t))}_{\LL^2(\Om)}\int_\Om\frac{\abs{\olj_\kappa(t)}^2}{m(\olu_\kappa(t))}\dd x\]
	and similarly for $\wtj_\kappa(t)$. Hence, invoking the Lipschitz continuity of $m$ and~\eqref{eq:aPrioriLIH1u}, we obtain
	\begin{align*}
		\norm{\olj_\kappa(t)}_{\LL^{4/3}(\Om)}^2+\norm{\wtj_\kappa(t)}_{\LL^{4/3}(\Om)}^2&\leq C_2\norm{m'}_\infty\Bigparen{\int_\Om\frac{\abs{\olj_\kappa(t)}^2}{m(\olu_\kappa(t))}+\frac{\abs{\wtj_\kappa(t)}^2}{m(\wtu_\kappa(t))}\dd x}\\
		&\leq C_2\norm{m'}_\infty\bigparen{\abs{\pa\calE}_{\calD_{\ulu_\kappa(t)}}^2(\olbfu_\kappa(t))+\abs{\pa\calE}_{\calD_{\ulu_\kappa(t)}}^2(\wtbfu_\kappa(t))}.
	\end{align*}
	Integrating over time $t\in[0,\infty)$, using the discrete EDB~\eqref{eq:discreteEDB} and slope estimate~\eqref{eq:slopeEstimate}, we obtain~\eqref{eq:aPrioriL2L43boldj}.
\end{proof}

\begin{corollary}[Discrete energy-dissipation inequality]
	For every $n\in\N_0$, $\kappa>0$, it holds
	\begin{align*}
		\calE(\boldu_\kappa^{n+1})+\int_{n\kappa}^{(n+1)\kappa}\int_\Om\frac12\biggbraces{&\Bigparen{\frac{\abs{\olj_\kappa}^2}{m(\olu_\kappa)}+\frac{1}{\tau(\ulu_\kappa)}\abs{\olz_\kappa-K(\olu_\kappa)}^2}\notag\\
			&+\Bigparen{\frac{\abs{\wtj_\kappa}^2}{m(\wtu_\kappa)}+\frac{1}{\tau(\ulu_\kappa)}\abs{\wtz_\kappa-K(\wtu_\kappa)}^2}}\dd x\dd t\leq\calE(\boldu_\kappa^n).
	\end{align*}
	From there, \eqref{eq:approxFreeEnergyEDI} is easily derived.
\end{corollary}

\begin{proof}
	From the discrete EDB~\eqref{eq:discreteEDB} and slope estimate~\eqref{eq:slopeEstimate} we know
	\begin{equation}
		\frac12\kappa\abs{\pa\calE}_{\calD_{u_\kappa^n}}^2(\boldu_\kappa^{n+1})+\frac12\int_{n\kappa}^{(n+1)\kappa}\abs{\pa\calE}_{\calD_{u_\kappa^n}}^2(\wtbfu_\kappa(t))\dd t\leq\calE(\boldu_\kappa^n)-\calE(\boldu_\kappa^{n+1}).
	\end{equation}
	Inserting~\eqref{eq:PcdiscreteFluxSlopeEstimate} and~\eqref{eq:discreteFluxSlopeEstimate} yields the claim.
\end{proof}

\section{Proof of Theorem \ref{th:mainTheorem}}\label{sec:mainTheorem}

In this section, we pass to the limit in the time-discrete scheme and prove Theorem \ref{th:mainTheorem}.
We first extract convergent subsequences, then verify the weak equations and energy-dissipation inequalities.

\begin{proposition}\label{pr:convergentSubsequences}
	For every sequence $(\kappaell)_{\ell\in\N}\downarrow0$, there exists a subsequence (non-relabelled) and a quadruple $(u,z,\olj,\wtj)$, such that the interpolants from Theorem~\ref{th:discreteSolutions} converge to $u$, $z$, $\olj$ and $\wtj$ in the following senses:
	\begin{subequations}\label{eq:convergences}
		\begin{alignat}{2}
			\olu_\kappaell,\wtu_{\kappa_\ell}&\weakto u&&\quad\text{in }\LL^2(I;\HH^2(\Om)),\label{eq:convergenceL2H2oluwtu}\\
			\olu_\kappaell,\ulu_\kappaell,\wtu_\kappaell&\to u&&\quad\text{in }\LL^2(I;\HH^1(\Om)),\label{eq:convergenceL2H1u}\\
			\olu_\kappaell,\ulu_\kappaell,\wtu_\kappaell&\weakto u&&\quad\text{in }\HH^1(\Om)\text{ pointwise }\forall t\in[0,\infty),\label{eq:convergencePtH1u}\\
			\olz_\kappaell,\ulz_\kappaell,\whz_\kappaell,\wtz_\kappaell&\starto z&&\quad\text{in }\LL^\infty([0,\infty);\LL^2(\Om)),\label{eq:convergenceLIL2z}\\
			\olz_\kappaell,\ulz_\kappaell,\whz_\kappaell,\wtz_\kappaell&\weakto z&&\quad\text{in }\LL^2(\Om)\text{ uniformly on }I,\label{eq:convergenceCwIL2z}\\
			\pa_t\whz_\kappaell&\weakto\pa_tz&&\quad\text{in }\LL^2([0,\infty);\LL^2(\Om)),\label{eq:convergenceL2L2dotz}\\
			\olj_\kappaell,\wtj_\kappaell&\weakto\olj,\wtj&&\quad\text{in }\LL^2(I;\LL^{4/3}(\Om;\R^d))\label{eq:convergenceL2L43Fluxes}
		\end{alignat}
	\end{subequations}
	on every compact interval $I\subset[0,\infty)$.
\end{proposition}

\begin{proof}
	The approximate equi-continuity~\eqref{eq:aPrioriCauchyWu} from Lemma~\ref{le:aPrioriEasy} yields for every $s,t\geq0$
	\[\limsup_{\kappa\to0}\calW_m(\olu_\kappa(s),\olu_\kappa(t))\leq\sqrt{2(\calE(\boldu^0)-\calEmin)}\sqrt{\abs{t-s}}.\]
	Moreover, by~\eqref{eq:aPrioriLIH1u}, all $\olu_\kappaell(t)$ stay inside one sublevel set of $\calE$. The sublevel sets of $\calE$ are compact in $(\calX_\frakm^M,\calW_m)$ by Lemma~\ref{le:metricProperties}~\ref{it:metricPropertySublevels}. Thus, a refined version of Ascoli--Arzel{\`a}~\cite[Proposition 3.3.1]{AmbrosioGigliSavare2008} provides a curve $u:[0,\infty)\to\calX_\frakm^M$, continuous with respect to $\calW_m$, such that, along a subsequence that we do not relabel, $\calW_m(\olu_\kappaell(t),u(t))\to0$ as $\ell\to\infty$ for every $t\in[0,\infty)$.
	The regularity and convergence improve by the a priori estimate~\eqref{eq:aPrioriLIH1u}, i.e.\ we have $\olu_\kappaell(t)\weakto u(t)$ in $\HH^1(\Om)$ for every $t\in[0,\infty)$ and $\olu_\kappaell\starto u$ in $\LL^\infty([0,\infty);\HH^1(\Om))$, and hence $u\in\CC_w([0,\infty);\calM_\frakm^+(\varcl{\Om}))\cap\LL^\infty([0,\infty);\HH^1(\Om))\hookrightarrow\CC_{w,b}([0,\infty);\HH^1(\Om))$. Invoking~\eqref{eq:aPrioriInterpolantsU}, the compact embedding $\HH^1(\Om)\overset{c}{\hookrightarrow}\LL^2(\Om)$ and Lebesgue's convergence theorem, we further obtain
	\begin{alignat*}{2}
		\olu_{\kappa_n},\ulu_{\kappa_n},\wtu_{\kappa_n}&\starto u&&\quad\text{in }\LL^\infty([0,\infty);\HH^1(\Om)),\\
		\olu_{\kappa_n}(t),\ulu_{\kappa_n}(t),\wtu_{\kappa_n}(t)&\weakto u(t)&&\quad\text{in }\HH^1(\Om),\quad\forall t\in[0,\infty),\\
		\olu_\kappaell,\ulu_\kappaell,\wtu_\kappaell&\to u&&\quad\text{in }\LL^2(I;\LL^2(\Om))
	\end{alignat*}
	for every compact time interval $I\subset[0,\infty)$. Taking into account the uniform $\LL^2(I;\HH^2(\Om))$-estimates~\eqref{eq:aPrioriL2H2olu&ulu} and~\eqref{eq:aPrioriL2H2wtu}, we deduce $u\in\Lloc^2([0,\infty);\HH^2(\Om))$ and~\eqref{eq:convergenceL2H2oluwtu}.
	We combine this with the strong $\LL^2(I;\LL^2(\Om))$-convergence and the interpolation Lemma~\ref{le:H1L2H2Interpolation}, which immediately yields~\eqref{eq:convergenceL2H1u} for $\olu_\kappaell$ and $\wtu_\kappaell$. The strong convergence for $\ulu_\kappaell$ follows by decomposing
	\[\norm{\ulu_\kappaell-u}_{\LL^2(I;\HH^1(\Om))}\leq\norm{\ulu_\kappaell-u}_{\LL^2([0,\kappaell]\cap I;\HH^1(\Om))}+\norm{\ulu_\kappaell-u}_{\LL^2(I\setminus[0,\kappaell];\HH^1(\Om))}.\]
	Then, \eqref{eq:aPrioriLIH1u} ensures $\norm{\ulu_\kappaell-u}_{\LL^2([0,\kappaell]\cap I;\HH^1(\Om))}\leq\sqrt{\kappaell}C_2\to0$, and the other term converges to zero by interpolation as previously.

	The existence of $z$, $\olj$ and $\wtj$ with the asserted convergences~\eqref{eq:convergenceLIL2z}, \eqref{eq:convergenceL2L2dotz} and~\eqref{eq:convergenceL2L43Fluxes} follow readily after passing, if necessary, to a further subsequence, from the a priori estimates \eqref{eq:aPrioriLIL2z}, \eqref{eq:aPrioriL2L2dotz}, \eqref{eq:aPrioriInterpolantsZ} and \eqref{eq:aPrioriL2L43boldj}. The last convergence, \eqref{eq:convergenceCwIL2z}, can be derived from
	\[\norm{\whz_\kappa(t)-\whz_\kappa(s)}_{\LL^2(\Om)}\leq\int_{s\wedge t}^{s\vee t}\norm{\dot\whz_\kappa(r)}_{\LL^2(\Om)}\dd r\leq C_2\sqrt{\abs{t-s}},\]
	where we used \eqref{eq:aPrioriL2L2dotz}. Thus, since $\norm{\whz_\kappa(t)}_{\LL^2(\Om)}\leq C_2$ uniformly in $\kappa$ and $t$, the weak version of Ascoli--Arzel\`a (Proposition~\ref{pr:ascoliWeakVersion}) and \eqref{eq:aPrioriInterpolantsZ} yields \eqref{eq:convergenceCwIL2z}.
\end{proof}

\begin{corollary}\label{co:StrongL2L2ConvergenceU}
	In the situation of Proposition~\ref{pr:convergentSubsequences}: For $\ell\to\infty$ we have
	\begin{subequations}\label{eq:NonlinearConvergencesU}
		\begin{alignat}{2}
			K(\olu_\kappaell)&\to K(u)&&\quad\text{in }\LL^2(I;\LL^2(\Om)),\label{eq:convergenceL2H1Ku}\\
			P(\olu_\kappaell)&\to P(u)&&\quad\text{in }\LL^1(I;\LL^1(\Om)),\label{eq:convergenceL1L1Pu}\\
			m(\olu_\kappaell)&\to m(u)&&\quad\text{in }\LL^2(I;\HH^1(\Om)).\label{eq:convergenceL2H1mu}
		\end{alignat}
	\end{subequations}
\end{corollary}

\begin{proof}
	Statements \eqref{eq:convergenceL2H1Ku} and \eqref{eq:convergenceL2H1mu} follow directly from \eqref{eq:convergenceL2H1u} and the fact that $K$ and $m$ are Lipschitz continuous.
	Moreover, \eqref{eq:convergenceL1L1Pu} is clear when $M<+\infty$ since $P$ is supposed to be continuous on $[0,M]$.
	In case $M=+\infty$, \eqref{eq:convergenceL1L1Pu} is obtained via Vitali's convergence theorem, similarly as in Step~3 of the proof of Lemma~\ref{le:freeEnergyAlongPotentialFlow}. More precisely, \eqref{eq:interpolationExponentq}, \eqref{eq:aPrioriLIH1u} and~\eqref{eq:aPrioriL2H2olu&ulu} ensure that $(P(u_\kappaell))_{\ell\in\N}$ is uniformly bounded in $\LL^1(I\times\Om)$. Furthermore, by \eqref{eq:asymptoticBehaviorP} and \eqref{eq:potentialDecompositionBounds}, for every $\ve>0$ there exists a constant $C_\ve\in(0,\infty)$ such that $\abs{P(s)}\leq\ve(1+s^\omega+\Fconv(s))$ for every $s\in[C_\ve,+\infty)$. Consequently, if $K\subset I\times\Om$ satisfies $\Leb^{d+1}(K)\leq\frac{\ve}{\norm{P}_{\CC([0,C_\ve])}}$, then
	\begin{align*}
		\int_K\abs{P(u_\kappaell)}\dd(t,x)&=\int_{\{u_\kappaell<C_\ve\}\cap K}\abs{P(u_\kappaell)}\dd(t,x)+\int_{\{u_\kappaell\geq C_\ve\}\cap K}\abs{P(u_\kappaell)}\dd(t,x)\\
		&\leq\ve+\ve\bigparen{\Leb^{d+1}(I\times\Om)+\norm{u_\kappaell}_{\LL^\omega(I\times\Om)}^\omega+\abs{I}(\calE(\boldu^0)-C_\calE)},
	\end{align*}
	showing that $(P(u_\kappaell))_{\ell\in\N}$ has uniformly absolutely continuous integrals.
\end{proof}

With these convergence results in hand, we now pass to the limit in the discrete equations and dissipation inequalities.

\noindent
\emph{Proof of Theorem~\ref{th:mainTheorem}.}
To keep notation simple, we continue writing $\kappa$ instead of $\kappaell$.\\

\noindent
\textbf{Step 1:} Equations \eqref{eq:equationU}, \eqref{eq:fluxCondition} and \eqref{eq:equationZ}.\\\noindent
We first pass to the limit in \eqref{eq:approxEquationU}. To this end, we fix $\psi\in\CC_c^1((0,\infty))$ and $V\in\CC^2(\varcl{\Om})$ with $\nabla V\cdot\n=0$ on $\boundary{\Om}$. Let $I\subset[0,\infty)$ be a compact interval such that $\supp(\psi)\subseteq I$. From \eqref{eq:convergenceL2H1u} it immediately follows
\begin{equation}\label{eq:equationUPrelim1}
	\int_0^\infty\int_\Om\psi'Vu\dd x\dd t=\lim_{\kappa\to0}\int_0^\infty\int_\Om\psi'V\olu_\kappa\dd x\dd t.
\end{equation}
For the right-hand side in~\eqref{eq:approxEquationU}, we use that \eqref{eq:convergenceL2H1mu} implies $\divv(m(\olu_\kappa)\nabla V)\to\divv(m(u)\nabla V)$ in $\LL^2(I;\LL^2(\Om))$, and \eqref{eq:assumptionA}, \eqref{eq:convergenceL2H1u}, \eqref{eq:convergenceL2H1Ku} yield $A(\olu_\kappa)(\olz_\kappa-K(\olu_\kappa))\weakto A(u)(z-K(u))$ in $\LL^2(I;\LL^2(\Om))$. Hence, together with \eqref{eq:convergenceL2H2oluwtu} and \eqref{eq:convergenceL1L1Pu}, we deduce $\boldN[\olbfu_\kappa,\nabla V]\weakto\boldN[\boldu,\nabla V]$ in $\LL^1(I)$. Therefore,
\begin{equation}\label{eq:equationUPrelim2}
	\int_0^\infty\psi\boldN[\boldu,\nabla V]\dd t=\lim_{\kappa\to0}\int_0^\infty\underline{\psi}_\kappa\boldN[\olbfu_\kappa,\nabla V]\dd t
\end{equation}
follows since $\underline{\psi}_\kappa\to\psi$ uniformly on $I$. Using~\eqref{eq:convergenceL2L43Fluxes} and treating $\boldN[\wtbfu_\kappa,\cdot]$ similarly as $\boldN[\olbfu_\kappa,\cdot]$, we pass to the limit in~\eqref{eq:approxFluxConditions} and obtain for every $\boldzeta\in\CC^1(\varcl{\Om};\mathbb R^d)$ with $\boldzeta\cdot\n=0$ on $\boundary{\Om}$, any $\psi\in\CC_c^0((0,\infty))$ 
\begin{equation}
	\int_0^\infty\int_\Om\psi\boldzeta\cdot\boldj\dd x\dd t=\int_0^\infty\psi\boldN[\boldu,\boldzeta]\dd t,
\end{equation}
showing~\eqref{eq:fluxCondition}. Combining this with~\eqref{eq:equationUPrelim1} and~\eqref{eq:equationUPrelim2}, we obtain~\eqref{eq:equationU}.

The ODE~\eqref{eq:equationZ} is derived easily from \eqref{eq:approxEquationZ}, the convergences \eqref{eq:convergenceL2H1u}, \eqref{eq:convergenceLIL2z}, \eqref{eq:convergenceL2L2dotz}, \eqref{eq:convergenceL2H1Ku} and the fact that $\tau$ is continuous and satisfies $\abs{\frac{1}{\tau}}\leq\frac{1}{\taumin}$.

The initial conditions $u(0)=u^0$ and $z(0)=z^0$ follow readily from \eqref{eq:convergencePtH1u} and \eqref{eq:convergenceCwIL2z}, i.e.\ $u^0=\olu_\kappaell(0)\weakto u(0)$ in $\HH^1(\Om)$ and $z^0=\olz_\kappaell(0)\weakto z(0)$ in $\LL^2(\Om)$.\\

\noindent
\textbf{Step 2:} Energy-dissipation inequality \eqref{eq:freeEnergyEDI}.\\\noindent
We pass to the limit in~\eqref{eq:approxFreeEnergyEDI}. Since $\olu_\kappa(t)\weakto u(t)$ in $\HH^1(\Om)$ and $\olz_\kappa(t)\weakto z(t)$ in $\LL^2(\Om)$ by \eqref{eq:convergencePtH1u} and \eqref{eq:convergenceCwIL2z}, the lower semicontinuity of $\calE$ from Lemma~\ref{le:freeEnergyProperties} implies
\[\calE(\boldu(t))\leq\liminf_{\kappa\to0}\calE(\olbfu_\kappa(t))\]
for every $t\in[0,\infty)$. The limit passage in the dissipation terms in~\eqref{eq:approxFreeEnergyEDI} relies on the convergences~\eqref{eq:convergenceL2H1u}, \eqref{eq:convergenceLIL2z}, and~\eqref{eq:convergenceL2L43Fluxes} and weak lower semicontinuity of the $\LL^2((0,t)\times\Om)$-norm.
At this point we emphasise that this argument needs $\frac{\olj_\kappaell}{\sqrt{m(\olu_\kappaell)}},\frac{\wtj_\kappaell}{\sqrt{m(\wtu_\kappaell)}}\weakto\frac{\boldj}{\sqrt{m(u)}}$ in $\LL^2((0,t)\times\Om)$. This convergence comes from \eqref{eq:approxFreeEnergyEDI} and a limit identification argument that uses that \eqref{eq:convergenceL2H1u} implies $\sqrt{m(\olu_\kappaell)},\sqrt{m(\wtu_\kappaell)}\to\sqrt{m(u)}$ in $\LL^4((0,t)\times\Om)$.\\

\noindent
\textbf{Step 3:} Strong convergence of $z$ under Assumption~\ref{as:entropyInequality}.\\\noindent
First, we derive a uniform bound of $\olz_\kappa$ in $\Lloc^\infty([0,\infty);\LL^\infty(\Om))$. Using the discrete ODE~\eqref{eq:approxEquationZ} and rearranging terms, we get for any $\kappa>0$, $k\in\N$
\[z_\kappa^k=\frac{\tau(u_\kappa^{k-1})}{\tau(u_\kappa^{k-1})+\kappa}z_\kappa^{k-1}+\frac{\kappa}{\tau(u_\kappa^{k-1})+\kappa}K(u_\kappa^k).\]
Hence, since $\abs{K(u)}\leq\Amax\abs{u}$, we infer for every $n\in\N$
\begin{align*}
	\norm{z_\kappa^n}_{\LL^\infty(\Om)}&\leq\sum_{k=1}^n\bigparen{\norm{z_\kappa^k}_{\LL^\infty(\Om)}-\norm{z_\kappa^{k-1}}_{\LL^\infty(\Om)}}+\norm{z^0}_{\LL^\infty(\Om)}\\
	&\leq\sum_{k=1}^n\frac{\kappa\Amax}{\taumin}\norm{u_\kappa^k}_{\LL^\infty(\Om)}+\norm{z^0}_{\LL^\infty(\Om)}=\frac{\Amax}{\taumin}\norm{\olu_\kappa}_{\LL^1([0,n\kappa];\LL^\infty(\Om))}+\norm{z^0}_{\LL^\infty(\Om)}.
\end{align*}
By Assumption~\ref{as:entropyInequality}, we either use $\abs{u_\kappa^k}\leq M<+\infty$, or, if $M=+\infty$, the Sobolev embedding $\HH^2(\Om)\hookrightarrow\LL^\infty(\Om)$ in $d\in\{1,2,3\}$. In any case, we have
\[\norm{\olz_\kappa}_{\LL^\infty([0,T];\LL^\infty(\Om))}\leq C(T)\]
for a constant $C(T)\in(0,\infty)$ depending on $T$ but not on $\kappa$. Next, we let $\kappa$ and $\sigma$ be two time steps. Exploiting the Lipschitz continuity of $\tau$ and $K/\tau$, we have on any compact interval $I\subset[0,\infty)$
\begin{align*}
	\frac12\ddd{t}\norm{\whz_\kappa-\whz_\sigma}_{\LL^2(\Om)}^2&=-\int_\Om\Bigparen{\frac{\olz_\kappa-K(\olu_\kappa)}{\tau(\ulu_\kappa)}-\frac{\olz_\sigma-K(\olu_\sigma)}{\tau(\ulu_\sigma)}}(\whz_\kappa-\whz_\sigma)\dd x\\
	&\leq-\int_\Om\frac{1}{\tau(\ulu_\kappa)}\Bigparen{(\olz_\kappa-\whz_\kappa)+(\whz_\kappa-\whz_\sigma)+(\whz_\sigma-\olz_\sigma)}(\whz_\kappa-\whz_\sigma)\dd x\\
	&\quad-\int_\Om\biggbraces{\olz_\sigma\Bigparen{\frac{1}{\tau(\ulu_\kappa)}-\frac{1}{\tau(\ulu_\sigma)}}-\Bigparen{\frac{K(\ulu_\kappa)}{\tau(\ulu_\kappa)}-\frac{K(\ulu_\sigma)}{\tau(\ulu_\sigma)}}}(\whz_\kappa-\whz_\sigma)\dd x\\
	&\quad+\int_\Om\biggbraces{\frac{K(\olu_\kappa)-K(\ulu_\kappa)}{\tau(\ulu_\kappa)}-\frac{K(\olu_\sigma)-K(\ulu_\sigma)}{\tau(\ulu_\sigma)}}(\whz_\kappa-\whz_\sigma)\dd x\\
	&\leq-\frac{1}{\taumax}\norm{\whz_\kappa-\whz_\sigma}_{\LL^2(\Om)}^2+\frac{C_2(\sqrt{\kappa}+\sqrt{\sigma})}{\taumin}\norm{\whz_\kappa-\whz_\sigma}_{\LL^2(\Om)}\\
	&\quad+\bigparen{\norm{\olz_\sigma}_{\LL^\infty(\Om)}\mathrm{Lip}(\tau^{-1})+\mathrm{Lip}(K/\tau)}\norm{\ulu_\kappa-\ulu_\sigma}_{\LL^2(\Om)}\norm{\whz_\kappa-\whz_\sigma}_{\LL^2(\Om)}\\
	&\quad+\frac{\Amax}{\taumin}\bigparen{\norm{\olu_\kappa-\ulu_\kappa}_{\LL^2(\Om)}+\norm{\olu_\sigma-\ulu_\sigma}_{\LL^2(\Om)}}\norm{\whz_\kappa-\whz_\sigma}_{\LL^2(\Om)}\\
	&\leq C_{10}(\kappa+\sigma+R_{\kappa,\sigma}(t)),
\end{align*}
where $R_{\kappa,\sigma}(t)=\norm{\ulu_\kappa(t)-\ulu_\sigma(t)}_{\LL^2(\Om)}^2+\norm{\olu_\kappa(t)-\ulu_\kappa(t)}_{\LL^2(\Om)}^2+\norm{\olu_\sigma(t)-\ulu_\sigma(t)}_{\LL^2(\Om)}^2$ and $C_{10}=C_{10}(I)\in(0,\infty)$ is independent of $\kappa$, $\sigma$ and $t$, but it may depend on $I$. Applying Gronwall's inequality yields
\[\sup_{t\in I}\norm{\whz_\kappaell(t)-\whz_{\kappa_n}(t)}_{\LL^2(\Om)}^2\leq C_{10}\int_I\bigparen{\kappaell+\kappa_n+R_{\kappaell,\kappa_n}(t)}\dd t\to0\]
for $\ell,n\to\infty$, where $\int_IR_{\kappaell,\kappa_n}(t)\dd t\to0$ due to the strong convergence~\eqref{eq:convergenceL2H1u}.\\

\noindent
\textbf{Step 4:} Entropy-dissipation inequality~\eqref{eq:internalEnergyEDI}.\\\noindent
We want to pass to the limit in~\eqref{eq:approxInternalEnergyEDI}. By~\eqref{eq:convergencePtH1u}, the lower semicontinuity of $\calU$ yields $\calU(u(t))\leq\liminf_{\kappa\to0}\calU(\olu_\kappa(t))$. We treat the three dissipation terms separately.
For the Laplacian term, we can simply use $\Delta\olu_\kappa\weakto\Delta u$ in $\LL^2([0,t];\LL^2(\Om))$ and weak lower semicontinuity of the norm.
For the coupling term, we additionally invoke the strong convergence $A(\olu_\kappa)(\olz_\kappa-K(\olu_\kappa))\to A(u)(z-K(u))$ in $\LL^2([0,t];\LL^2(\Om))$.
For the potential term, we follow Step~4 of the proof of Lemma~\ref{le:freeEnergyAlongHeatFlow}, now on the time-space domain $(0,t)\times\Om$. This means, we decompose $F''=\Fconv''+\Fconc''$, and the convex part is treated using Fatou's lemma, while the concave part is handled by Vitali's convergence theorem. Note that, in the static argument, the $\HH^2(\Om)$-bound gives $\nabla\olu_\kappa$ a spatial integrability exponent strictly larger than 2. In the time-space setting, the corresponding uniform bound follows instead by interpolation~\eqref{eq:interpolationExponentq}, which ensures $\nabla\olu_\kappa\in\LL^\omega((0,t)\times\Om)$ with $\omega>2$.\\

The proof of Theorem~\ref{th:mainTheorem} is finished.\hfill$\Box$

\section*{Conclusion}

We constructed solutions for a degenerate VPS system using minimising movement schemes based on the product $\calW_m\times\LL^2$-topology.
By construction, the solutions preserve the structural properties of the system already at the level of the time-discrete approximants, including the physical bounds on the phase-field variable and the sharp energy-dissipation inequality in terms of the PDE fluxes.
The proof of the EDI is based on the De\,Giorgi interpolant.
The flux interpolants are obtained as Riesz representatives using an operator bound in terms of the metric slope.
A coercivity estimate for the metric slope is established to derive a uniform $\LL^2(0,T;\HH^2(\Om))$-bound for the De\,Giorgi interpolant of the phase-field variable, enabling the limit passage when the time step tends to zero.

Under stronger data assumptions, we also established the entropy-dissipation inequality, providing a promising starting point for future work on weak--strong uniqueness via relative entropy methods.

Finally, we note that for regularity and the EDI, flow interchange can in principle be avoided using slope coercivity estimates for both interpolants. The present analysis, however, still relies on flow interchange in order to derive a weak formulation of the equation. As in~\cite{LisiniMatthesSavare2012}, to ensure the EVI$_\lambda$ property of the regularised potential
flow, we had to restrict to mobilities that degenerate linearly at the endpoints.
It would therefore be interesting to find a method that allows to infer the weak formulation directly from the variational scheme without relying on flow interchange.

\appendix

\section{Appendix}

\subsection{Heat equation and nonlinear drift-diffusion equation}

The following two well-posedness results can be shown with standard parabolic theory.

\begin{proposition}
	Assume $\Om\subseteq\R^d$ is a bounded domain, convex or with $\CC^{1,1}$-boundary, and $v^0\in\HH^1(\Om)$. Then the heat equation with homogeneous Neumann boundary conditions~\eqref{eq:heatEquation} admits a unique strong solution
	\[v\in\Hloc^1([0,\infty);\LL^2(\Om))\cap\CC([0,\infty);\HH^1(\Om))\cap\Lloc^2([0,\infty);\HH^2(\Om)).\]
	In particular, $v$ satisfies the following two energy-dissipation identities for every $t\in[0,\infty)$
	\begin{align}
		\frac12\int_\Om\abs{v(t)}^2\dd x+\int_0^t\int_\Om\abs{\nabla v(s)}^2\dd x\dd s&=\frac12\int_\Om\abs{v^0}^2\dd x,\label{eq:L2NormDissipationOfHeatEquation}\\
		\frac12\int_\Om\abs{\nabla v(t)}^2\dd x+\int_0^t\int_\Om\abs{\Delta v(s)}^2\dd x\dd s&=\frac12\int_\Om\abs{\nabla v^0}^2\dd x.\label{eq:DirichletEnergyDissipationOfHeatEquation}
	\end{align}
\end{proposition}

\begin{proposition}\label{pr:driftDiffusion}
	Assume $\Om\subseteq\R^d$ is a bounded domain, convex or with $\CC^{1,1}$-boundary, $\boldzeta\in\WW^{1,\infty}(\Om;\R^d)$ with $\boldzeta\cdot \n =0$ on $\boundary{\Om}$ and $v^0\in\HH^2(\Om)$ with $\nabla v^0\cdot\n=0$ on $\boundary{\Om}$ and $0\leq v_0\ (\leq M<+\infty)$. Then the nonlinear drift-diffusion equation~\eqref{eq:viscousPotentialEquation} has a unique strong solution
	\begin{equation}\label{eq:driftDiffusionRegularity}
		v\in\Hloc^2([0,\infty);(\HH^1(\Om))^*)\cap\Hloc^1([0,\infty);\HH^1(\Om))\cap\CC([0,\infty);\HH^2(\Om)).
	\end{equation}
	The weak maximum principle ensures $0\leq v_t\ (\leq M)$ for every $t\in[0,\infty)$, and the regularity~\eqref{eq:driftDiffusionRegularity} immediately implies $\pa_tv,\divv(m(v)\boldzeta)\in\CC([0,\infty);\LL^2(\Om))$.
\end{proposition}

\begin{lemma}[Strong maximum principle]\label{le:strongMaximumPrinciple}
	Let $\Om\subset\R^d$ be a bounded domain with a $\CC^l$-boundary for some $l\in\N$ satisfying
	\begin{equation}\label{eq:BoundaryRegularityMaximumPrinciple}
		l>\frac{(d+2)^2+4}{2(d+2)}=2+\frac{d^2}{2(d+2)}.
	\end{equation}
	Let $v$ be the solution to~\eqref{eq:heatEquation} for initial data $v^0\not\equiv\text{constant}$. Then $v\in\CC^{0,1}((0,\infty)\times\varcl{\Om})$ (continuous in time, continuously differentiable in space once), and for every $t>0$
	\[\essinf_{x\in\Om}v^0(x)<\min_{x\in\varcl{\Om}}v(t,x)\leq\max_{x\in\varcl{\Om}}v(t,x)<\esssup_{x\in\Om}v^0(x).\]
\end{lemma}

\begin{proof}
	The claim follows from \cite[Appendix F, Proposition 52.7 with (52.11)]{QuittnerSouplet2019} and \cite[Theorem 10.4]{BesovIlinNikolskii1978} once we have shown $u\in\Wloc^{1,2;q}((0,T]\times\varcl{\Om})\cap\Lloc^\infty((0,T]\times\varcl{\Om})$ for every $T>0$ and some $q>d+2$, where the $\Wloc^{1,2;q}$-space denotes the set of functions that are weakly differentiable in time, twice weakly differentiable in space, with all derivatives belonging to $\Lloc^q((0,T]\times\varcl{\Om})$.

	The regularity assumption on $\boundary{\Om}$ guarantees $v\in\CC^\infty((0,\infty);\HH^l(\Om))$. In particular, $\pa_{x_ix_j}^2u\in\CC^\infty([\tau,\infty);\HH^{l-2}(\Om))$ for every $1\leq i,j\leq d$. Due to~\eqref{eq:BoundaryRegularityMaximumPrinciple}, $\HH^{l-2}(\Om)$ embeds into $\LL^{d+2+\delta}(\Om)$ for some $\delta>0$. Indeed, this is clear if $d\leq2(l-2)$, and whenever $d>2(l-2)$, condition~\eqref{eq:BoundaryRegularityMaximumPrinciple} is equivalent to $d+2<\frac{2d}{d-2(l-2)}$. Moreover, from
	\[l>\frac{(d+2)^2+4}{2(d+2)}=1+\frac{d}{2}+\frac{2}{d+2}>1+\frac{d}{2}\]
	it follows $d<2(l-1)$, i.e.\ $\HH^{l-1}(\Om)\hookrightarrow\LL^\infty(\Om)$. Thus, $u\in\Lloc^\infty((0,\infty)\times\varcl{\Om})$.
\end{proof}

\subsection{Functional analytic results and Sobolev inequalities}

\begin{proposition}[Ascoli--Arzel\`a, weakly compact version, cf.\ {\cite[Lemma C.1]{Lions1996}}]\label{pr:ascoliWeakVersion}
	Let $X$ be a reflexive Banach space, $X^*$ its dual, $Y\subseteq X^*$ a countable dense subset and $I\subset\R$ a compact interval. With $\CC_w(I;X)$ we denote the set of all weakly continuous functions, i.e.
	\[\CC_w(I;X):=\bigset{f:I\to X}{f(t_n)\weakto f(t)\text{ in }X\text{ whenever }t_n\to t}.\]
	Assume a subset $\mathcal{A}\subset\CC_w(I;X)$ satisfies the following two conditions:
	\begin{enumerate}
		\item $\sup_{f\in\mathcal{A},t\in I}\norm{f(t)}_X<\infty$;
		\item for every $\varphi\in Y$, the family $\{I\to\R,t\mapsto\product{\varphi}{f(t)}_{X^*,X}\}_{f\in A}$ is equicontinuous.
	\end{enumerate}
	Then $\mathcal{A}$ is relatively weakly compact, i.e.\ any sequence $(f_n)_{n\in\N}\subset\mathcal{A}$ admits a subsequence and a limit $f\in\CC_w(I;X)$ such that $\product{\varphi}{f_{n_k}(\cdot)}\to\product{\varphi}{f(\cdot)}$ uniformly on $I$ for every $\varphi\in X^*$.
\end{proposition}

For the next lemmas, suppose $\Om\subset\R^d$ is a bounded domain, being convex or having a $\CC^{1,1}$-boundary.

\begin{lemma}
	There exists a constant $C_\Delta\in(0,\infty)$ only depending on $\Om$ such that
	\begin{equation}\label{eq:H2DeltaEstimate}
		\biggnorm{u-\avint{_\Om}u\dd x}_{\HH^2(\Om)}\leq C_\Delta\norm{\Delta u}_{\LL^2(\Om)}
	\end{equation}
	for all $u\in\HHn^2(\Om)$.
\end{lemma}

This is a consequence of the a priori estimate from elliptic regularity (e.g.\ \cite[Chapters II, III]{Grisvard2011}) $\norm{u}_{\HH^2(\Om)}\leq C(\norm{\Delta u}_{\LL^2(\Om)}+\norm{u}_{\LL^2(\Om)})$ and Poincar\'e-Wirtinger's inequality $\norm{u-\avint{_\Om}u\dd x}_{\LL^2(\Om)}\leq\CPW\norm{\nabla u}_{\LL^2(\Om;\R^d)}\leq\CPW^2\norm{\Delta u}_{\LL^2(\Om)}$.

\begin{lemma}\label{le:H1L2H2Interpolation}
	Every $v\in\HHn^2(\Om)$ satisfies $\norm{v}_{\HH^1(\Om)}\leq\sqrt[4]{d{+}1}\norm{v}_{\LL^2(\Om)}^{1/2}\norm{v}_{\HH^2(\Om)}^{1/2}$.
\end{lemma}

\begin{proof}
	By an easy calculation that uses integration-by-parts and the boundary condition we find
	\begin{align*}
		\norm{v}_{\HH^1(\Om)}^2&=\int_\Om\bigparen{v^2+\nabla v\cdot\nabla v}\dd x=\int_\Om v(v-\Delta v)\dd x\\
		&\leq\norm{v}_{\LL^2(\Om)}\biggparen{\norm{v}_{\LL^2(\Om)}+\sum_{i=1}^d\norm{\pa_{x_ix_i}^2v}_{\LL^2(\Om)}}\\
		&\leq\sqrt{d+1}\norm{v}_{\LL^2(\Om)}\biggparen{\norm{v}_{\LL^2(\Om)}^2+\sum_{i=1}^d\norm{\pa_{x_ix_i}^2v}_{\LL^2(\Om)}^2}^{1/2}
	\end{align*}
\end{proof}

\printbibliography[heading=bibintoc]
\end{document}